\documentclass[10pt]{amsart}

\usepackage{enumerate}
\usepackage{amsmath,amssymb}
\usepackage{color}

\newtheorem{theorem}{Theorem}[section]

\newtheorem{corollary}[theorem]{Corollary}
\newtheorem{lemma}[theorem]{Lemma}
\newtheorem{proposition}[theorem]{Proposition}

\theoremstyle{definition}
\newtheorem{example}[theorem]{Example}
\newtheorem{nonexample}[theorem]{Non-example}
\newtheorem{remark}[theorem]{Remark}

\newtheorem{question}[theorem]{Question}
\newtheorem{problem}[theorem]{Problem}

\def\A{{\mathcal{A}}}

\def\E{{\mathcal{E}}}
\def\P{{\mathcal{P}}}

\def\cR{{\mathcal{R}}}
\def\S{{\mathcal{S}}}

\def\C{{\mathbb{C}}}

\def\M{{\mathbb{M}}}
\def\N{{\mathbb{N}}}
\def\R{{\mathbb{R}}}

\def\Z{{\mathbb{Z}}}

\begin{document}
\title[Regular $^*$-subalgebras of $B(H)$]{On regular $^*$-algebras of bounded linear operators:
A new approach towards a theory of noncommutative Boolean algebras}

\author[M. Mori]{Michiya Mori}

\address{Interdisciplinary Theoretical and Mathematical Sciences Program (iTHEMS), RIKEN, 2-1 Hirosawa, Wako, Saitama 351-0198 Japan.}
\email{michiya.mori@riken.jp}

\thanks{The author is supported by RIKEN Special Postdoctoral Researcher Program.}


\date{}

\begin{abstract}
We study (von Neumann) regular $^*$-subalgebras of $B(H)$, which we call R$^*$-algebras.
The class of R$^*$-algebras coincides with that of ``E$^*$-algebras that are pre-C$^*$-algebras'' in the sense of   Z. Sz\H{u}cs and B. Tak\'acs.
We give examples, properties and questions of R$^*$-algebras.
We observe that the class of unital commutative R$^*$-algebras has a canonical one-to-one correspondence with the class of Boolean algebras. 
This motivates the study of R$^*$-algebras as that of noncommutative Boolean algebras.
We explain that seemingly unrelated topics of functional analysis, like AF C$^*$-algebras and incomplete inner product spaces, naturally arise in the investigation of R$^*$-algebras. 
We obtain a number of results on R$^*$-algebras by applying various famous theorems in the literature.
\end{abstract}
\maketitle
\thispagestyle{empty}

\section{Introduction}\label{intro}
Throughout the paper{,} $H$ is a complex Hilbert space{,} and $B(H)$ denotes the von Neumann algebra of bounded linear operators on $H$.  
%
A motivation for our research comes from a problem raised by Kadison. 
In the celebrated conference in Baton Rouge in 1967, Kadison proposed a list of 20 problems on operator algebras. 
The list was shared among operator algebraists worldwide, which brought a great influence to the development of the theory of operator algebras over the last half-century. 
In 2003, Ge published a survey article \cite{Ge} on Kadison's list of problems.
Among the 20 problems, 19 are about C$^*$- or von Neumann algebras. 
The only exception is listed as the 15th problem in Ge's article, which asks the following.
\begin{problem}[Kadison]\label{kadison}
If a self-adjoint operator algebra is finitely generated (algebraically) and each self-adjoint operator in it has finite spectrum, is it finite-dimensional?
\end{problem}
In the recent paper \cite{ST}, Sz\H{u}cs and Tak\'acs discovered that for a $^*$-subalgebra $A$ of $B(H)$ several conditions are equivalent to the condition every self-adjoint element of $A$ has finite spectrum. 
In this paper, we first observe in Theorem \ref{equi} that these conditions are also equivalent to the condition $A$ is regular in the sense of von Neumann. 
We call a $^*$-subalgebra of $B(H)$ that satisfies these equivalent conditions an R$^*$-algebra (``R'' meaning ``regular'').
The abundance of the conditions for a $^*$-subalgebra of $B(H)$ to be an R$^*$-algebra implies that R$^*$-algebras form more or less an important class of $^*$-algebras. 

Even though we assume no completeness, it turns out that R$^*$-algebras share some important properties with C$^*$-/von Neumann algebras. 
For example, we may freely use the technique of functional calculus.
Recall that there exists a correspondence between commutative C$^*$-algebras and locally compact Hausdorff spaces (Gelfand--Naimark theorem), and similarly, between commutative von Neumann algebras and certain measure spaces.
Is there such a correspondence for R$^*$-algebras?
By von Neumann's study of $^*$-regular rings in \cite[Part II]{N}, projections of a general R$^*$-algebra form a lattice.
It follows that the lattice of projections of a commutative R$^*$-algebra is a generalized Boolean algebra.
In fact, there exists a complete correspondence between the class of commutative R$^*$-algebras and that of generalized Boolean algebras. 
By restricting ourselves to the unital case, we see that unital commutative R$^*$-algebras correspond to Boolean algebras. 
In that sense (the projection lattice of) an R$^*$-algebra can be considered as a noncommutative (generalized) Boolean algebra. 
This correspondence is closely tied with Stone's theorem on the duality between Boolean algebras and Boolean spaces. 

Are there plenty of noncommutative R$^*$-algebras?
The answer is yes. 
The class of ultramatricial R$^*$-algebras, that is, directed unions of finite-dimensional C$^*$-algebras, provides us with many examples of noncommutative R$^*$-algebras. 
It turns out that Problem \ref{kadison}, which remains to be open, is equivalent to asking whether every R$^*$-algebra is ultramatricial. 
Ultramatricial R$^*$-algebras with countable Hamel dimension are highly relevant to the theory of separable AF C$^*$-algebras, whose classification by Elliott \cite{E} is one of the best{-}known results of the theory of C$^*$-algebras. 
Another subclass of ultramatricial R$^*$-algebras, which we call purely atomic R$^*$-algebras, is constructed from inner product spaces. 
We give a complete classification of purely atomic R$^*$-algebras in terms of linear isometries of inner product spaces.
By applying Banach space theory and results on operator ranges, we prove that there are many mutually nonisomorphic purely atomic R$^*$-algebras. 

The literature of regular rings, Boolean algebras, and inner product spaces in connection with Banach spaces, as well as C$^*$-/von Neumann algebras provides us with many deep results and further questions. 
The paper intends to collect some of them in order to give supporting {evidence} for the author's belief that the study of R$^*$-algebras is interesting enough to be considered in more depth. 
This paper should give much motivation to the research on Problem \ref{kadison}, which seems to have gotten less attention among operator algebraists compared to other problems in the list \cite{Ge}. 
The choice of materials in this paper is more or less biased by the author's preference.

This paper is formulated in the following manner. 
In the next section{,} we give the definition and basic properties of R$^*$-algebras. 
In particular, we characterize R$^*$-algebras by closedness of the ranges of operators, the uniqueness of C$^*$-norm of $^*$-subalgebras as well as von Neumann regularity (Theorems \ref{equi}, \ref{c*norm}).
Section \ref{example} provides examples of R$^*$-algebras and discusses Problem \ref{kadison}. 
The correspondence between commutative R$^*$-algebras and Boolean algebras is given in Section \ref{com}. 
Section \ref{atom} is devoted to giving the relation between purely atomic R$^*$-algebras and inner product spaces. 
Some results that arise by comparing the projection lattice of an R$^*$-algebra with that of a von Neumann algebra or a Baer $^*$-ring are given in Section \ref{comparison}.
The consequence of Elliott's classification theorem of AF C$^*$-algebras is summarized in Section \ref{countable}. 
We also revisit Elliott's theorem from a lattice theoretic viewpoint (Theorem \ref{sumlattice}).
In Section \ref{nd} we collect results that derive from von Neumann, Dye and Gleason's theorems. 
We close this paper by proposing questions on R$^*$-algebras for further research (Section \ref{question}). 

In the paper{,} we frequently use basic results on C$^*$-/von Neumann algebras as in \cite{KR1,KR2,Ta}, and von Neumann's theory on complemented modular lattices and regular rings as in (the first four chapters of) \cite[Part II]{N}. 
None of the proofs is dependent on any results by Sz\H{u}cs and Tak\'acs in \cite{ST}, although the underlying idea may have some overlap implicitly.

\section{Regular $^*$-subalgebras of $B(H)$}\label{defi}
If there is no confusion, for each $z\in \C$, we often write $z$ instead of the scalar multiplication operator $z\cdot \operatorname{id}_{H}\in B(H)$.
In particular, $1$ means the unit of $B(H)$.
An \emph{algebra} means an algebra over the field of complex numbers, possibly without unit.
A subspace or a subalgebra in this paper need not be closed. 
For a subset $S\subset B(H)$, the symbol $\operatorname{span} S$ means the linear space spanned by $S$, that is, the smallest complex-linear subspace of $B(H)$ that contains $S$. 
Similarly, $\A(S)$ and $\A^*(S)$ denote the algebra and the $^*$-algebra generated by $S$, respectively.
Namely, $\A(S)$ (resp.\ $\A^*(S)$) is the smallest subalgebra (resp.\ $^*$-subalgebra) of $B(H)$ that contains $S$.
For a finite subset $\{x_1, \ldots, x_n\}\subset B(H)$, we define $\A(x_1, \ldots, x_n):=\A(\{x_1, \ldots, x_n\})$ and $\A^*(x_1, \ldots, x_n):=\A^*(\{x_1, \ldots, x_n\})$.
A \emph{projection} of $B(H)$ is an operator $p\in B(H)$ with $p=p^2=p^*$.

Let $x\in B(H)$. 
The range of $x$ is written as $\operatorname{ran} x$ or $xH$, and the kernel as $\ker x$.
The symbol $\sigma(x)$ stands for the spectrum of $x$, i.e., $\sigma(x):=\{z\in \C\mid z - x \text{ is not invertible in }B(H)\}$.
Let $l(x)$ and $r(x)$ denote the left and right support projections of $x$. 
That is, $l(x)$ is the projection onto the closure of $\operatorname{ran} x$, and $r(x)$ is the projection onto $(\ker x )^{\perp}$.
If in addition $x$ is self-adjoint, then $s(x):=l(x)=r(x)$ is called the \emph{support} of $x$.
Recall that every $x\in B(H)$ has the unique polar decomposition $x=v\lvert x\rvert$ such that $\lvert x\rvert = (x^*x)^{1/2}$ and $vv^*=l(x)$, $v^*v=r(x)$.

Although the theorem below can be obtained as a corollary of \cite[Theorems 3.8 and 4.1]{ST} by Sz\H{u}cs and Tak\'acs, our discovery of (or at least the emphasis on) its relation with von Neumann regularity is new. 
This viewpoint motivates a deeper study on this class of algebras in connection with Boolean algebras in the subsequent sections.
For the convenience of the reader, we give a proof {that} is independent of \cite{ST}. 
Our proof needs no background beyond the standard first course of functional analysis. 

\begin{theorem}\label{equi}
Let $A$ be a $^*$-subalgebra of $B(H)$. The following are equivalent.
\begin{enumerate}
\item $A$ is regular (in the sense of von Neumann), i.e., for each $x\in A$ there exists $y\in A$ with $xyx=x$. 
\item $A$ is algebraic, i.e., for each $x\in A$ the algebra $\A(x)$ generated by $x$ has finite dimension.
\item Each operator in $A$ has closed range. 
\item Each operator in $A$ has finite spectrum. 
\item Each self-adjoint operator in $A$ has closed range. 
\item Each self-adjoint operator in $A$ has finite spectrum.
\end{enumerate}
\end{theorem}

Before proving it, let us first give a preliminary result. 
Let $x, y\in B(H)$. 
Then $y$ is called a \emph{Moore--Penrose inverse} of $x$ if $y$ satisfies $l(y)=r(x)$, $r(y)= l(x)$, $yx=r(x)$ and $xy = l(x)$.
It is easy to see that such an operator (if there exists) is unique.   

\begin{lemma}
For an operator $x\in B(H)$ the following are equivalent. 
\begin{itemize}
\item The range of $x$ is closed. 
\item There exists $y\in B(H)$ with $xyx=x$. 
\item $x$ has a Moore--Penrose inverse.
\end{itemize}
\end{lemma}
\begin{proof}
Let $x\in B(H)$ have closed range. 
Then $x$ restricts to a bounded bijection $x_0$ from $(\ker x)^{\perp}$ onto $\operatorname{ran} x$, so $x_0$ has a bounded inverse by the inverse mapping theorem. 
We define $y$ as the unique operator in $B(H)$ with $y=x_0^{-1}$ on $\operatorname{ran} x$ and $y=0$ on $(\operatorname{ran} x)^{\perp}$.
Then we have $xyx=x$ by a direct calculation.
If $xyx=x$, then we have $x=xyx= xr(x)yl(x)x$, and it is easy to show that the operator $r(x)yl(x)$ is the Moore--Penrose inverse of $x$.
If $x$ has a Moore--Penrose inverse, the restriction of $x$ as a mapping from $(\ker x)^{\perp}$ into the closure of $\operatorname{ran} x$ has a bounded inverse, which implies the closedness of $\operatorname{ran} x$. 
\end{proof}

\begin{proof}[Proof of Theorem \ref{equi}]
The implications $(2)\Rightarrow(4)\Rightarrow(6)$ and  $(3)\Rightarrow(5)$ are obvious.
The above lemma implies $(1)\Rightarrow (3)$.\\
$(5)\Rightarrow(6)$  
Recall that a self-adjoint operator $a\in B(H)$ has closed range if and only if $0$ is not an accumulation point of the spectrum $\sigma(a)$ of $a$. 
Assume that a self-adjoint operator $a\in A$ has infinite spectrum. 
Take an accumulation point $s\in \R$ in the spectrum $\sigma(a)$. 
By the spectral mapping theorem, the self-adjoint operator $a^2-sa\in A$ has $0$ as an accumulation point of its spectrum, so the range of $a^2-sa$ is not closed.\\  
$(6)\Rightarrow(1)$ Let $x\in A$ and $x=v\lvert x\rvert$ be the polar decomposition of $x$ in $B(H)$. 
Since $x^*x$ is a self-adjoint operator with finite spectrum, we may take a polynomial $P_1$ with real coefficients that coincides with the square root map on $\sigma(x^*x)\cup \{0\}$. 
We thus have $\lvert x\rvert = (x^*x)^{1/2} = P_1(x^*x) \in A$.  
We may also take a polynomial $P_2$ with real coefficients such that $P_2(t)=t^{-1}$ for every $t\in \sigma(\lvert x\rvert)\setminus\{0\}$ and $P_2(0)=0$. 
We thus obtain $v= v\lvert x\rvert P_2(\lvert x\rvert) = xP_2(\lvert x\rvert)\in A$. 
Set $y:= P_2(\lvert x\rvert) v^*\in A$. 
Then we obtain $xyx = v\lvert x\rvert P_2(\lvert x\rvert) v^*v\lvert x\rvert = v\lvert x\rvert =x$. \\
$(6)\Rightarrow(2)$ Let $x\in A$. 
Take self-adjoint operators $a, b\in A$ with $x=a+ib$, and put $A_0 := \A(a, b, 1) = \A^*(a, b, 1)\subset A+\C \operatorname{id}_H\subset B(H)$. 
Observe that $A_0$ has countable Hamel dimension.
Fix $s\in \R$.
For $z\in \C$, set $f(z) :=\exp(sa+zb)\in B(H)$. 
Then $f\colon \C\to B(H)$ is a holomorphic map.
If $t\in \R$, then $sa+tb$ has finite spectrum, and we obtain $f(t) = \exp (sa+tb) \in \A^*(sa+tb, 1) \subset A_0$.
Take a Hamel basis $\{e_n\}_{n\geq 1}$ of $A_0$. 
For each $n\geq 1$, the subset $F_n:=\{t\in \R\mid f(t)\in \mathrm{span}\{e_1, \ldots, e_n\}\}$ is closed in $\R$. 
Moreover, we have $\cup_{n\geq 1}F_n = \R$ because $f(\R)\subset A_0$.
By the Baire category theorem, we see that some $F_n$ has nonempty interior $U$.
Take a bounded linear surjective operator $\pi\colon B(H) \to \mathrm{span}\{e_1, \ldots, e_n\}$ such that $\pi^2=\pi$.
We see that $\pi\circ f=f$ on $U$. 
On the other hand, both $f$ and $\pi\circ f$ are holomorphic maps from $\C$ into $B(H)$.
By the identity theorem, we obtain $\pi\circ f = f$ on $\C$. 
Thus $f(\C)\subset \mathrm{span}\{e_1, \ldots, e_n\}$.
It follows that $\exp(sx) = \exp (s(a+ib))  \in A_0$ for every $s\in \R$. 
Define the holomorphic map $g\colon \C\to B(H)$ by $g(z)=\exp(zx)$.
By the same reasoning as above, we see that the image of $g$ is contained in a finite-dimensional subspace of $A_0$. 
This plainly forces $x$ to be algebraic. 
\end{proof}

We call $A$ that satisfies the equivalent conditions of Theorem \ref{equi} a \emph{regular $^*$-subalgebra} of $B(H)$, or an \emph{R$^*$-algebra}, for short. 
The same object is called an ``E$^*$-algebra that is a pre-C$^*$-algebra'' in \cite{ST}.

\begin{corollary}\label{subalg}
Every $^*$-subalgebra of an R$^*$-algebra is an R$^*$-algebra.
\end{corollary}
\begin{proof}
Clear from any item of  Theorem \ref{equi} (2)--(6).
\end{proof}

The equivalence of (1)--(4) of the above theorem does not hold in general if we consider an algebra instead of a $^*$-algebra. 
Indeed, it is known that there exists a \emph{closed} subalgebra $A$ of $B(H)$ such that (3) and (4) hold but (2) does not hold \cite[Theorems 5 and 6]{W}.  
Since $\{x\in B(H)\mid x^*x=0\} =\{0\}$, we obtain
\begin{lemma}\label{starregular}
An R$^*$-algebra $A$ is a $^*$-regular ring, i.e., $A$ is a $^*$-ring that is regular and satisfies $\{x\in A\mid x^*x=0\} =\{0\}$. 
\end{lemma}

The following lemma is a consequence of an elementary fact about algebraic elements.
\begin{lemma}\label{ax1}
Let $x\in B(H)$ be an algebraic operator (i.e., $\A(x)$ is finite-dimensional). 
If $x$ is invertible in $B(H)$, then $x$ is invertible in $\A(x, 1)$. 
\end{lemma}
\begin{proof}
Assume that $x$ is an algebraic operator that is invertible in $B(H)$.
Then the minimal polynomial $P_1(t)=\sum_{j=0}^k c_jt^j$ of $x$ must satisfy $c_0\neq0$. 
Indeed, if $c_0=0$, then the polynomial $P_2(t)=P_1(t)/t$ satisfies $P_2(x)=x^{-1}P_1(x)=0$, contradicting the minimality of $P_1$.
Thus $1=(c_0-P_1(x))/c_0 = -\sum_{j=1}^k c_jx^j/c_0$, which clearly implies that $x$ is invertible in $\A(x, 1)$. 
\end{proof}

Although we cannot make use of tools that involve completeness (like taking limits) in the setting of R$^*$-algebras in general, we may take advantage of a number of techniques from operator theory.

\begin{proposition}\label{operation}
Let $A\subset B(H)$ be an R$^*$-algebra, and $x\in A$. 
\begin{enumerate}
\item Assume that $A$ contains the unit of $B(H)$. 
Let $U\subset \C$ be an open set with $\sigma(x)\subset U$. Let $f\colon U\to \C$ be a holomorphic map. 
Then the operator $f(x)$ (which is obtained by holomorphic functional calculus in $B(H)$) belongs to $A$.  
\item Assume that $x$ is normal ($xx^*=x^*x$). Let $f\colon \sigma(x)\to \C$ be a function with $f(0)=0$.
Then the operator $f(x)$ (which is obtained by Borel functional calculus in $B(H)$) belongs to $A$.
In particular, if $x$ is self-adjoint, then the positive part $x_+:=f_+(x)$ and the negative part $x_-:=f_-(x)$ belong to $A$, where $f_+(t) := \max\{0, t\}$, $f_-(t):=\max\{-t, 0\}$, $t\in \R$. 
\item Let $x=v\lvert x\rvert$ be the polar decomposition of $x$ in $B(H)$. Then $v, \lvert x\rvert\in A$, hence $l(x), r(x)\in A$.  
\item The Moore--Penrose inverse $x^{\dagger}$ of $x$ belongs to $A$.
\end{enumerate}
\end{proposition}
\begin{proof}
(1) Let $x\in A$ and $z\in \C$. 
If $x-z$ is invertible in $B(H)$, the preceding lemma implies that $(x-z)^{-1}\in \A(x, 1)$. 
Since $\A(x, 1)$ is finite-dimensional, it is a closed subalgebra of $B(H)$.  
Thus it is clear by the definition of holomorphic functional calculus that $f(x)\in \A(x, 1)\subset A$.\\
(2) Since $\sigma(x)$ is a finite set, $f(x)$ coincides with some polynomial in $x$ and $x^*$. Thus we obtain $f(x)\in \A^*(x)\subset A$.\\
(3), (4) See Proof of Theorem \ref{equi} $(6)\Rightarrow(1)$. 
The Moore--Penrose inverse of $x$ is the $y$ constructed there.
\end{proof}

We call a linear mapping (resp.\ a linear bijection) $\phi\colon A\to B$ between two $^*$-algebras a \emph{$^*$-homomorphism} (resp.\ a \emph{$^*$-isomorphism}) provided that it is multiplicative and $^*$-preserving, i.e., $\phi(xy)=\phi(x)\phi(y)$ and $\phi(x^*)=\phi(x)^*$ for any $x, y\in A$. 

In view of the theory of C$^*$-/von Neumann algebras, we also give an abstract axiom for a $^*$-algebra to be $^*$-isomorphic to an R$^*$-algebra.
Let $A$ be a $^*$-algebra (over $\C$). 
Recall that a \emph{C$^*$-norm} of $A$ is a norm $\lVert\cdot \rVert\colon A\to [0, \infty)$ satisfying $\lVert xy\rVert\leq \lVert x\rVert\lVert y\rVert$ (submultiplicativity) and $\lVert x^*x\rVert = \lVert x\rVert^2$ (C$^*$-identity) for all $x, y\in A$. 
A \emph{pre-C$^*$-algebra} is a $^*$-algebra endowed with a C$^*$-norm. 
It is clear that every $^*$-subalgebra of $B(H)$ is a pre-C$^*$-algebra.

The (noncommutative) Gelfand--Naimark theorem \cite[Theorem I.9.18]{Ta} states that a complete pre-C$^*$-algebra is $^*$-isomorphic to a closed $^*$-subalgebra of $B(H)$. 
On the other hand, Sakai's theorem \cite[Theorem III.3.5]{Ta} states that a C$^*$-algebra $A$ is $^*$-isomorphic to a von Neumann algebra if and only if $A$ is a dual Banach space.  
\begin{theorem}
Let $A$ be a pre-C$^*$-algebra. 
Then the following conditions are mutually equivalent. 
\begin{itemize}
\item  $A$ is $^*$-isomorphic to an R$^*$-algebra. 
\item $A$ is regular. 
\item $A$ is algebraic.
\end{itemize}
\end{theorem}
\begin{proof}
The completion of $A$ is a C$^*$-algebra, which is $^*$-isomorphic to a closed $^*$-subalgebra of some $B(H)$ by the Gelfand--Naimark theorem. Now the equivalence easily follows from Theorem \ref{equi}.
\end{proof}
We call a pre-C$^*$-algebra with the above equivalent conditions an \emph{abstract R$^*$-algebra}.
(This can also be referred to as an \emph{algebraic pre-C$^*$-algebra}, or a \emph{regular pre-C$^*$-algebra}.)
As is the case for C$^*$-algebras, the above theorem ensures that in many situations there is no need to distinguish between abstract R$^*$-algebras and R$^*$-algebras in $B(H)$ (see also Proposition \ref{isometry} below, which implies the uniqueness of C$^*$-norm of an R$^*$-algebra).

\begin{remark}
In this paper, an element of a pre-C$^*$-algebra is said to be \emph{positive} if it is positive in the completion (as an operator of a C$^*$-algebra).
\end{remark}

In what follows we collect some basic properties of ideals, quotients and $^*$-homomorphisms of R$^*$-algebras, which were first proved in \cite{ST}. We give a proof for the reader's convenience.
Our proof is based on C$^*$-algebra theory. 
\begin{proposition}[{\cite{ST}}]\label{ideal}
Let $I$ be a (two-sided) ideal of an R$^*$-algebra $A$. 
Then $I$ is closed under $^*$-operation and in the norm topology in $A$, and the quotient by the ideal forms an R$^*$-algebra. 
\end{proposition}
\begin{proof}
Let $I\subset A$ be an ideal. 
Let $x\in I$ and let $x=v\lvert x\rvert$ be its polar decomposition. 
Then $v, \lvert x\rvert\in A$ by Proposition \ref{operation} (3). 
Thus we have $I\ni v^*v\lvert x\rvert=\lvert x\rvert$ and $x^*= \lvert x\rvert v^*\in I$.

Let $x_n\in I$, $n\geq 1$, converge to $x\in A$. 
Then $x_n^*x_n\to x^*x$, hence 
\[
y_n:=\lvert x\rvert^{\dagger} x_n^*x_n\lvert x\rvert^{\dagger} \to \lvert x\rvert^{\dagger}x^*x\lvert x\rvert^{\dagger} = s(\lvert x\rvert)
\]
as $n\to \infty$. 
Since $y_n$ is positive and $s(y_n)\leq s(\lvert x\rvert^\dagger)=s(\lvert x\rvert)$, the above convergence implies that $s(y_n)= s(\lvert x\rvert)$ for a sufficiently large $n$. 
For each $n$, we have $y_n= \lvert x\rvert^{\dagger} x_n^*x_n\lvert x\rvert^{\dagger}\in I$ and hence $s(y_n)=y_ny_n^{\dagger}\in I$. 
It follows that $s(\lvert x\rvert) = s(y_n)\in I$ for a sufficiently large $n$, thus we have $x=xs(\lvert x\rvert) \in I$.
Therefore,  $I$ is closed in $A$.

Consider the completion $\widehat{A}$ of $A$, which is a C$^*$-algebra. 
It is easy to see that the closure $J$ of $I$ in $\widehat{A}$ is a closed ideal of $\widehat{A}$ and $J\cap A=I$. 
Let $\pi$ denote the quotient map $\widehat{A}\to \widehat{A}/J$. 
Then the restriction of $\pi$ to $A$ has kernel $I$, and the image $\pi(A)$ is endowed with the C$^*$-norm of $\widehat{A}/J$. 
This implies that $A/I$ forms a pre-C$^*$-algebra, which is clearly algebraic.
\end{proof}

Recall that there are C$^*$-algebras with no projections (see for example \cite[Sections IV.8 and VII.8]{Da}).
In contrast, every R$^*$-algebra has plenty of projections because of Theorem \ref{equi} (6). 
This is somehow akin to von Neumann algebras, so we introduce the following concept borrowed from the theory of von Neumann algebras.

\begin{proposition}[Reduced/induced R$^*$-algebras (cf.\ {\cite[Definition II.3.11]{Ta}})]
Let $A\subset B(H)$ be an R$^*$-algebra.
\begin{enumerate}
\item If $p\in A$ is a projection, then $pAp =\{x\in A\mid pxp=x\}$ forms an R$^*$-algebra in $B(H)$ with unit $p$.
\item If $p'\in A'=\{ y\in B(H)\mid xy=yx\text{ for all }x\in A\}$ is a projection in the commutant of $A$, then $p'A = \{p'x\mid x\in A\}\subset B(H)$ forms an R$^*$-algebra.
\end{enumerate}
\end{proposition}
\begin{proof} 
Easy.
Just check that $pAp$ and $p'A$ form $^*$-subalgebras of $B(H)$ and that each operator in them has closed range.
\end{proof}

If $A\subset B(H)$ is an R$^*$-algebra and $A$ has a unit $e$, then $e$ need not coincide with the unit of $B(H)$. 
However, it is easy to check that $e$ is a projection and $A=eAe$ can be considered as a $^*$-subalgebra of $B(eH)$, and the embedding $A=eAe\to B(eH)$ preserves the unit. 
With this in mind, when we consider a unital R$^*$-algebra in $B(H)$, we may always assume that $1=\operatorname{id}_H\in A$. 
This enables us to consider functional calculus as in Proposition \ref{operation} (2) for functions $f$ without assuming $f(0)=0$ if $A$ is unital.

Let $A\subset B(H)$ be an R$^*$-algebra without unit. 
Let $\widetilde{A}$ denote the algebra $A+\C\cdot\operatorname{id}_H\subset B(H)$, which is called the \emph{unitization} of $A$. 
By Theorem \ref{equi} (4), $\widetilde{A}$ is a unital R$^*$-algebra.

\begin{proposition}[{\cite{ST}}]\label{isometry}
Let $A$ be an R$^*$-algebra and $K$ be a Hilbert space. 
\begin{enumerate}
\item Every injective $^*$-homomorphism $\phi\colon A\to B(K)$ is an isometry. 
\item For every $^*$-homomorphism $\phi\colon A\to B(K)$ the kernel of $\phi$ is an ideal of $A$. 
\item Let $\pi$ be the canonical projection of $A$ onto $A/\ker \phi$. 
There exists a unique injective $^*$-homomorphism $\phi_0\colon A/\ker \phi\to B(K)$ with $\phi=\phi_0\circ\pi$. 
\end{enumerate}
\end{proposition}
\begin{proof}
$(1)$ Let $\phi\colon A\to B(K)$ be an injective $^*$-homomorphism. 
Clearly the image $\phi(A)$ is an R$^*$-algebra. 
Let $A$ act on the Hilbert space $H$. 
If $A$ is unital, then we may assume that $\operatorname{id}_H\in A$ and $\phi(\operatorname{id}_H)=\operatorname{id}_K$ by the comment preceding this proposition. 
If $A$ is not unital, then $\phi$ uniquely extends to an injective  $^*$-homomorphism $\widetilde{\phi}\colon \widetilde{A}\to B(K)$ by the formula $\widetilde{\phi}(x+z\operatorname{id}_H)= \phi(x)+z\operatorname{id}_K$, $x\in A$, $z\in \C$. 
Therefore, it suffices to consider the case $A$ is unital and $\phi(\operatorname{id}_H)=\operatorname{id}_K$. 
It follows from Lemma \ref{ax1} that for each $x\in A$, $x$ is invertible (in $\A(x, \operatorname{id}_H)$, or equivalently, in $B(H)$) if and only if $\phi(x)$ is invertible (in $\A(\phi(x), \operatorname{id}_K)$, or equivalently, in $B(K)$).
Therefore, for each positive operator $a\in A$, we obtain 
$\lVert a\rVert = \sup\sigma(a)=\sup\sigma(\phi(a)) = \lVert \phi(a)\rVert$. 
This implies that for every $x\in A$, 
\[
\lVert x\rVert^2=\lVert x^*x\rVert = \lVert \phi(x^*x)\rVert = \lVert \phi(x)^*\phi(x)\rVert =\lVert \phi(x)\rVert^2,
\]
thus $\phi$ is an isometry.\\
The rest of the statement is a consequence of a routine straightforward argument combined with Proposition \ref{ideal}.
\end{proof}

A (pre-C)$^*$-algebra $A$ is said to \emph{have a unique C$^*$-norm} if there exists a unique norm of $A$ that makes $A$ a pre-C$^*$-algebra. 
Recall that every C$^*$-algebra has a unique C$^*$-norm \cite[Corollary I.5.4]{Ta}. 
The preceding proposition shows that every R$^*$-algebra has a unique C$^*$-norm.
More generally, it is easy to see that a $^*$-subalgebra $A\subset B(H)$ has a unique C$^*$-norm if it satisfies the following property: 
\begin{itemize}
\item[(LC$^*$)] The closure of the ($^*$-)subalgebra $\A(a)$ in $B(H)$ is contained in $A$ for every positive element $a\in A$. 
\end{itemize}
(Proof: Let $\lVert \cdot\rVert$ denote the operator norm of $B(H)$, and let $\lVert\cdot\rVert_1$ be any C$^*$-norm of $A$.
If $a\in A$ is positive, then the closure of $\A(a)$ in $B(H)$ is a C$^*$-algebra that is a subalgebra of $A$. 
Since a C$^*$-algebra has a unique C$^*$-norm, we have $\lVert a\rVert=\lVert a\rVert_1$. 
For a general $x\in A$, by the C$^*$-identity and the positivity of $x^*x$ we obtain $\lVert x\rVert^2=\lVert x^*x\rVert =\lVert x^*x\rVert_1=\lVert x\rVert_1^2$.)
In particular, the union of any increasing net of C$^*$-algebras has a unique C$^*$-norm.

In \cite{ST} Sz\H{u}cs and Tak\'acs introduced the following property (E$^*$) for a $^*$-algebra $A$, and proved that it is equivalent to some conditions in Theorem \ref{equi} in a general setting. 
\begin{itemize}
\item[(E$^*$)] For any $^*$-subalgebra $B\subset A$, any Hilbert space $K_1$, any $^*$-homomorphism $\pi\colon B\to B(K_1)$ and any vector $v\in K_1$ such that $\pi(B)v$ is dense in $K_1$, there are a Hilbert space $K_2$, a $^*$-homomorphism $\pi_0\colon A\to B(K_2)$ and a vector $w\in K_2$ such that $\pi_0(A)w$ is dense in $K_2$ and $\langle \pi(x)v, v\rangle =\langle \pi_0(x)w, w\rangle$ for every $x\in B$. 
\end{itemize}
Let us deduce the equivalence of the conditions in Theorem \ref{equi} and (E$^*$) for pre-C$^*$-algebras without resorting to results of \cite{ST}. 
We also clarify that these conditions are related to the uniqueness of C$^*$-norm. 

\begin{theorem}\label{c*norm}
Let $A\subset B(H)$ be a $^*$-subalgebra. 
The following are equivalent. 
\begin{enumerate}
\item $A$ is an R$^*$-algebra. 
\item Every $^*$-subalgebra of $A$ has a unique C$^*$-norm. 
\item For every positive element $a\in A$ the $^*$-subalgebra $\A(a)=\A^*(a)\subset A$ has a unique C$^*$-norm. 
\item $A$ satisfies $(E^*)$. 
\end{enumerate}
\end{theorem}
\begin{proof}
Proposition \ref{isometry} combined with Corollary \ref{subalg} shows $(1)\Rightarrow(2)$, and $(2)\Rightarrow(3)$ is clear.\\ 
$(3)\Rightarrow (1)$ Assume that there is a self-adjoint operator $a\in A$ with infinite spectrum. 
By replacing $a$ with $a^2$ we may assume that $a$ is positive.
We prove that $\A^*(a)$ has more than one C$^*$-norm. 
Let $A_0$ denote the algebra of polynomials $P(t) = \sum_{k=1}^n z_kt^k$ with $n\geq 1$, $z_k\in \C$.
It is clear that $A_0\ni P(t)\mapsto P(a)\in \A^*(a)$ is a $^*$-isomorphism.
The algebra of polynomials $A_0$ has more than two C$^*$-norms, thus $\A^*(a)$ has the same property.
Indeed, define e.g.\ $\lVert P(t)\rVert_j :=\sup_{\lambda\in[j-1, j]}\lvert P(\lambda)\rvert\in [0, \infty)$, $j=1, 2$. 
Then $\lVert \cdot\rVert_1$ and $\lVert \cdot\rVert_2$ are distinct C$^*$-norms of $A_0$.\\
$(1)\Rightarrow(4)$ Let $B\subset A$ be a $^*$-subalgebra, $K_1$ a Hilbert space, $\pi\colon B\to B(K_1)$ a $^*$-homomorphism, $v\in K_1$ a vector and assume that $\pi(B)v$ is dense in $K_1$. 
Corollary \ref{subalg} combined with Proposition \ref{isometry} implies that the mapping $f\colon B\to \C$ defined by $f(x):= \langle \pi(x)v, v\rangle$ is a bounded positive linear functional. 
It follows by the Hahn--Banach theorem (see \cite[Lemma III.3.2]{Ta}) that $f$ extends to a bounded positive linear functional $g$ on $A$. 
By the Gelfand--Naimark--Segal (GNS) construction \cite[Theorem I.9.14]{Ta} there are a Hilbert space $K_2$, a $^*$-homomorphism $\pi_0\colon A\to B(K_2)$ and a vector $w\in K_2$ such that $\pi_0(A)w$ is dense in $K_2$ and $g(x)=\langle \pi_0(x)w, w\rangle$ for every $x\in A$.
It follows that $\langle \pi(x)v, v\rangle =f(x)=g(x)=\langle \pi_0(x)w, w\rangle$ for every $x\in B$.\\ 
$(4)\Rightarrow(1)$ Assume that there is a self-adjoint operator $a\in A$ with infinite spectrum. 
We prove that $\A^*(a)$ does not satisfy (E$^*$), which clearly implies that $A$ does not satisfy (E$^*$), either.
Let $B$ denote the $^*$-subalgebra of $\A^*(a)$ generated by $a^2$.
Take a $^*$-homomorphism $\pi\colon B\to \C$ that is uniquely determined by $\pi(a^2)=-1$. 
(Note that $a$ has infinite spectrum and so does $a^2$, thus $a^2$ is not an algebraic operator.) 
Assume that there are a Hilbert space $K$, a $^*$-homomorphism $\pi_0\colon \A^*(a)\to B(K)$ and a vector $w\in K$ such that $\pi_0(\A^*(a))w$ is dense in $K$ and $\pi(x)=\langle \pi_0(x)w, w\rangle$ for every $x\in B$. 
Then $b:=\pi_0(a)$ is a self-adjoint operator in $B(K)$ such that 
\[
-1=\pi(a^2)= \langle \pi_0(a^2)w, w\rangle=\langle b^2w, w\rangle=\lVert bw\rVert^2\geq 0,
\]
which is a contradiction.
\end{proof}

\section{Examples}\label{example}
Although the definition of R$^*$-algebra might seem somehow restrictive, we can construct various examples of R$^*$-algebras.  
In this section{,} we give a number of examples of R$^*$-algebras, as well as {means} of obtaining new examples from known ones. 
More examples can be found in Sections \ref{com} and \ref{atom}.

First of all, for each $n\geq 1$, the $n\times n$ \emph{full matrix algebra} $\M_n=\M_n(\C)$ is the collection of all $n\times n$ complex matrices, which is an example of a finite-dimensional C$^*$-algebra and also an R$^*$-algebra.

\begin{example}[Finite-dimensional R$^*$-algebras]
Every finite-dimensional pre-C$^*$-algebra is a complete R$^*$-algebra (thus a C$^*$-algebra). 
In fact, a finite-dimensional (pre-)C$^*$-algebra is the direct sum of finitely many full matrix algebras. 
\end{example}

\begin{nonexample}
An infinite-dimensional C$^*$-algebra is never an R$^*$-algebra. 
Indeed, it is well-known that every infinite-dimensional C$^*$-algebra has an operator with infinite spectrum. See for example \cite[Exercise 4.6.14]{KR1}.
\end{nonexample}
It follows that the intersection of C$^*$- and R$^*$-algebras consists of finite-dimensional C$^*$-algebras.

\begin{example}[Sequence space and direct sums]
Let $c_{00}$ denote the algebra of complex sequences with finitely many nonzero terms. 
This acts on the sequential $\ell_2$ space by termwise multiplication, and it is clear that $c_{00}$ is an R$^*$-algebra in $B(\ell_2)$.

More generally, let $I$ be a set, and let $A_i\subset B(H_i)$ be an  R$^*$-algebra for each $i\in I$. 
Let $H$ be the direct sum of $H_i$, and let $p_i$ be the projection of $H$ onto $H_i$.
Each operator $x\in B(H_i)$ determines an operator $p_i xp_i\in B(H)$.
Then
\[
\bigoplus_{i\in I} A_i := \A^*(\{p_ixp_i\mid x\in A_i,\,\,i\in I\}) \subset B(H)
\] 
forms an R$^*$-algebra.
\end{example}

\begin{example}[Algebras of operators with finite rank]\label{compact}
Let  $F(H)$ denote the algebra of bounded linear operators with finite rank on a Hilbert space $H$.
Then $F(H)$ is an R$^*$-algebra in $B(H)$.
On the other hand, the collection $\M_{\infty}$ of infinite matrices $(a_{ij})_{i, j\geq 1}$ with finitely many nonzero entries forms an R$^*$-algebra, which acts on the sequential $\ell_2$ space in the usual manner.
It is clear that $\M_{\infty}\subset F(\ell_2)$. 
Let us observe that these are nonisomorphic.
Indeed, it is clear that $\M_{\infty}$ has countable Hamel dimension. 
However, $F(\ell_2)$ does not. 
(To prove it, consider for example the subspace of $F(\ell_2)$ that consists of the collection of operators $x$ with $x^*(\ell_2)\subset \C h$ for a fixed unit vector $h\in \ell_2$. 
This subspace is linearly isometric to $\ell_2$ by the map $x\mapsto xh$. 
Recall that $\ell_2$ has uncountable Hamel dimension.)
It is worth mentioning that both of the completions of $\M_{\infty}$ and $F(\ell_2)$ are the C$^*$-algebra $K(\ell_2)$ of compact operators. 
In Section \ref{atom} we will take a closer look at examples of this type.
\end{example}

\begin{example}[Directed unions]\label{af}
Let $I$ be a directed set and $(A_i)_{i\in I}$ be an increasing net of abstract R$^*$-algebras. 
Since every injective $^*$-homomorphism between R$^*$-algebras is isometric, the union $A=\bigcup_{i\in I} A_i$ forms an R$^*$-algebra.  
We say such an $A$ is \emph{ultramatricial} provided that each $A_i$ is finite-dimensional. 
If in addition $(A_i)_{i\in I}$ is a sequence (rather than a net), then we say $A$ is \emph{sequentially ultramatricial}.
\end{example}
In this paper{,} a C$^*$-algebra that arises as the completion of an ultramatricial R$^*$-algebra is called an \emph{AF} (approximately finite-dimensional) C$^*$-algebra.  
A C$^*$-algebra that arises as the completion of a sequentially ultramatricial R$^*$-algebra is called a \emph{sequentially AF} C$^*$-algebra. 
It is well-known that a separable AF C$^*$-algebra is sequentially AF  (see \cite[Theorem III.3.4]{Da}).  
CAUTION: In C$^*$-algebra theory it is common to consider only separable C$^*$-algebras, so an AF C$^*$-algebra often means a separable one in the literature, which is different from our choice of terminology.

\begin{proposition}[Matrix algebras]
Let $A$ be an R$^*$-algebra.  
For $n\geq 1$, the algebra $\M_n(A) = A\otimes \M_n$ of the collection of $n\times n$ matrices with entries in $A$ forms an R$^*$-algebra.  
\end{proposition}
\begin{proof}
It is well-known that the algebraic tensor product $A\otimes \M_n$ is a regular ring for every unital regular ring $A$. 
See for example \cite[Theorem II.2.13]{N}. 
If  $A\subset B(H)$ is a unital R$^*$-algebra, then $A\otimes \M_n$ can be considered as a $^*$-subalgebra of $B(H^{\oplus n})$. 
Thus $A\otimes \M_n$ forms an R$^*$-algebra.
If $A$ is nonunital, consider $A\otimes \M_n$ as a subalgebra of $\widetilde{A}\otimes \M_n$. 
\end{proof}
In \cite[Proof of Proposition 4.7]{ST}, the above proposition was shown by means of the fact that the matrix algebra over an algebraic algebra is algebraic \cite[Theorem 9]{A}.

\begin{example}[Tensor products with ultramatricial algebras]
Let $A$, $B$ be R$^*$-algebras.
Suppose that $B$ is ultramatricial. 
Take a directed set $I$ and an increasing net $(B_i)_{i\in I}$ of finite-dimensional R$^*$-algebras with $B=\bigcup_{i\in I} B_i$. 
Since a finite-dimensional R$^*$-algebra decomposes into a direct sum of matrix algebras, the preceding proposition implies that for each $i$ the algebraic tensor product $A\otimes B_i$ forms an R$^*$-algebra.
Moreover, for any $i, j\in I$ with $i\leq j$ we may take a canonical inclusion from $A\otimes B_i$ into $A\otimes B_j$ (as algebraic tensor products). 
We define $A\otimes B$ to be the union of the increasing net $(A\otimes B_i)_{i\in I}$.
\end{example}

In particular, for an R$^*$-algebra $A$, $\M_{\infty}(A) = A\otimes \M_{\infty}$ is an R$^*$-algebra, which can be considered as the collection of $\N\times\N$ matrices with finitely many nonzero entries in $A$. 
This R$^*$-algebra is used to consider the $K_0$-group in Section \ref{countable}. 

\begin{example}[Infinite tensor products]
Let $I$ be a set and $(A_i)_{i\in I}$ be a family of unital ultramatricial R$^*$-algebras. 
Then the above argument implies that for each finite subset $F$ of $I$ we may construct the tensor product R$^*$-algebra $\bigotimes_{i\in F}A_i$. 
Moreover, if $F_1\subset F_2\subset I$ are finite subsets, there exists a natural embedding $\iota\colon\bigotimes_{i\in F_1}A_i\to \bigotimes_{i\in F_2}A_i$ determined by the condition $\iota(\bigotimes_{i\in F_1} x_i) =(\bigotimes_{i\in F_1} x_i) \otimes (\bigotimes_{j\in F_2\setminus F_1} 1_{A_j})$. 
Thus the family $(\bigotimes_{i\in F}A_i)_F$ forms an increasing net of R$^*$-algebras, whose union over all finite subsets $F\subset I$ will be written as $\bigotimes_{i\in I}A_i$.
\end{example}

\begin{example}[Group algebras]
Let $G$ be a discrete group and consider the group algebra $\C[G]$. 
As usual, we endow $\C[G]$ with the $^*$-operation that is determined by $(zg)^*= \overline{z} g^{-1}$, $z\in \C$, $g\in G$.
Since the study of C$^*$-/von Neumann algebras arising from $\C[G]$ plays a dominant role in the theory of operator algebras, it is natural to ask when $\C[G]$ forms an R$^*$-algebra.
Recall that $G$ is said to be \emph{locally finite} if every finite subset generates a finite subgroup of $G$.
It is easy to see that the group algebra $\C[G]$ forms an R$^*$-algebra if and only if $G$ is locally finite. 
Indeed, if $\C[G]$ forms an R$^*$-algebra, then it is algebraic, which plainly leads to local finiteness of $G$. 
(For a finite subset $F\subset G$, consider the algebra generated by $\sum_{g\in F} g$.) 
The converse implication is even easier: 
The left regular representation gives a C$^*$-norm of $\C[G]$, and it is clear that every element in $\C[G]$ is algebraic if $G$ is locally finite. 
More generally, in \cite[Theorem 6.5]{ST} it was proved that the convolution $^*$-algebra of complex-valued compactly supported continuous functions on a locally compact group $G$ is an R$^*$-algebra if and only if $G$ is discrete and locally finite. 
Von Neumann regularity of more general group algebras (not necessarily over $\C$) is discussed in \cite[Problem 8]{Kap1}.
\end{example}

\subsection{Problems of Kurosh and Kadison}
In \cite{ST}, Sz\H{u}cs and Tak\'acs found that Kadison's Problem \ref{kadison} is related to Kurosh's problem. 
Let us revisit it. 
Recall that an algebra $A$ is defined to be \emph{locally finite-dimensional} if every finite subset of $A$ generates a finite-dimensional subalgebra of $A$. 
The classical Kurosh's problem asks whether every algebraic algebra $A$ (i.e., such that {every} single element of $A$ generates a finite-dimensional subalgebra of $A$) is locally finite-dimensional. 
This problem has a negative answer in the general setting. 
However, it is known that for certain classes of algebras (e.g., the class of PI-algebras), Kurosh's problem has a positive answer. 
See Kaplansky's list of problems \cite{Kap1} for references and several related topics.  
We would like to consider Kurosh's problem for the class of R$^*$-algebras. 
Since an R$^*$-algebra is algebraic by Theorem \ref{equi}, this is equivalent to asking whether every R$^*$-algebra is locally finite-dimensional.
This problem clearly coincides with Kadison's Problem \ref{kadison}.

\begin{lemma}
An R$^*$-algebra $A$ is locally finite-dimensional if and only if $A$ is ultramatricial.
\end{lemma}
\begin{proof}
Let $A$ be locally finite-dimensional. 
Consider the directed set $\mathcal{F}$ made of all finite subsets of $A$, ordered by inclusion. 
Then $A$ is the union of the increasing net $(\A(F\cup F^*))_{F\in \mathcal{F}}$ of finite-dimensional R$^*$-algebras.
Conversely, let $A$ be equal to the union of an increasing net $(A_i)_{i\in I}$ of finite-dimensional R$^*$-subalgebras. 
For each finite subset $F\subset A$, we may find some $i\in I$ such that $F\subset A_i$, and thus $\A(F)\subset A_i$ is 
finite-dimensional. 
\end{proof}

Therefore, Problem \ref{kadison}{,} as well as Kurosh's problem for R$^*$-algebras{,} is equivalent to asking whether every R$^*$-algebra is ultramatricial.

\section{Lattice of projections as noncommutative Boolean algebras}\label{com}
Let us recall some definitions of lattice theory for beginners. 
A \emph{lattice} is a partially ordered set $\P$ such that every pair $p, q\in \P$ has the least upper bound $p\vee q\in \P$ and the greatest lower bound $p\wedge q\in \P$. 
A mapping (resp.\ bijection) $\psi$ between two lattices is called a \emph{lattice homomorphism} (resp.\ a \emph{lattice isomorphism}) if it satisfies $\psi(p\vee q)=\psi(p)\vee\psi(q)$ and $\psi(p\wedge q)=\psi(p)\wedge\psi(q)$ for every pair of elements $p, q$.
A lattice $\P$ with greatest element $1$ and least element $0$ is \emph{complemented} if for every $p\in \P$ there exists $q\in \P$ with $p\vee q=1$ and $p\wedge q=0$. Such a $q$ is called a \emph{complement} of $p$.
A lattice $\P$ is \emph{modular} if $p\vee(q\wedge r)=(p\vee q)\wedge r$ for all $p, q, r\in \P$ with $p\leq r$.

An \emph{orthocomplemented lattice} is a complemented lattice $\P$ with a bijection $\P\ni p\mapsto p^{\perp}\in \P$ (called the orthocomplementation) with the following properties. 
\begin{itemize}
\item $(p^{\perp})^{\perp} = p$ for all $p\in\P$.
\item $p^{\perp}$ is a complement of $p$ for all $p\in\P$. 
\item $p\leq q\iff p^{\perp}\geq q^{\perp}$ for all $p, q\in\P$.
\end{itemize}
An \emph{orthoisomorphism} $\psi$ between two orthocomplemented lattices is a lattice isomorphism  such that $\psi(p^{\perp}) = \psi(p)^{\perp}$ for all $p$.

A family $\emptyset\neq\S\subset 2^{X}$ of subsets of a set $X$ is called a \emph{field of sets} if $A\cap B$, $A\cup B$, $A^{c}\in \S$ for any $A, B\in \S$. 
A field of sets forms an orthocomplemented lattice relative to the order by inclusion and the orthocomplementation by set-complement.
A \emph{Boolean algebra} is an orthocomplemented lattice that is orthoisomorphic to a field of sets.
It is clear that each element of a Boolean algebra has a unique complement. 
It follows that a lattice isomorphism between two Boolean algebras is an orthoisomorphism. 
The book \cite{GH} is a gentle introduction to the theory of Boolean algebras.
A \emph{generalized field of sets} is a family $\emptyset\neq\S\subset 2^{X}$ of subsets of a set $X$ such that $A\cap B$, $A\cup B$, $A\setminus B \in\S$ for any $A, B\in \S$. 
A lattice that is lattice isomorphic to a generalized field of sets is called a \emph{generalized Boolean algebra}.

Recall that for a von Neumann algebra $M$ the collection $\P(M)$ of projections forms an orthocomplemented lattice.
It is well-known that the projection lattice of a von Neumann algebra is modular if and only if the algebra is of finite type. 
Indeed, if a von Neumann algebra $M$ is not finite, then $M$ has $B(\ell_2)$ as a $\sigma$-weakly closed $^*$-subalgebra, but $\P(B(\ell_2))$ is not modular by \cite[Solution 14]{Ha}. 
The proof of modularity of $\P(M)$ when $M$ is of finite type can be found in e.g.\ \cite[Theorem V.1.37]{Ta}.

Let $A$ be an R$^*$-algebra with unit. 
By von Neumann's theory in \cite[Part II]{N}, the collection of principal right ideals of $A$, ordered by inclusion, forms a complemented modular lattice. 
Moreover, by \cite[Theorem II.4.5]{N} together with Lemma \ref{starregular}{,} this lattice can be identified with the lattice of projections $\P(A)=\{p\in A\mid p=p^2=p^*\}$.
Here $\P(A)$ is endowed with the usual order, namely, for two projections $p, q\in B(H)$, $p\leq q$ means $pH\subset qH$, or equivalently, $pq=p$.
The correspondence between principal right ideals and projections can be described in the following manner: 
Each principal right ideal of $A$ coincides with $pA$ for a unique $p\in \P(A)$.
Observe that $p^{\perp}:= 1-p$, $p\in A$, determines an orthocomplementation of $\P(A)$. 

Let $A$ be a unital R$^*$-algebra in $B(H)$. 
Then each element in $\P(A)$ can be identified with its range $pH$, which is a closed subspace of the Hilbert space $H$. 
The lattice structure of $\P(A)$ can be written more explicitly:

\begin{proposition}
Let $A$ be a unital R$^*$-algebra in $B(H)$ with $\operatorname{id}_H\in A$. 
For $p, q\in \P(A)$, 
\begin{itemize}
\item $p\vee q\in \P(A)$ is the projection onto $pH+ qH$; 
\item $p\wedge q\in \P(A)$ is the projection onto $pH\cap qH$; and 
\item $p^{\perp}\in \P(A)$ is the projection onto the orthogonal complement $(pH)^{\perp}=\{h\in H\mid \langle h, k\rangle=0\text{ for all }k\in pH\}$. 
\end{itemize}
\end{proposition}
\begin{proof}
Let us prove the first item. 
The projection $p\vee q$ is determined as the smallest projection in $A$ with $p\leq p\vee q$, $q\leq p\vee q$. 
Therefore, it is clear that the range of $p\vee q$ contains $pH+qH$. 
It suffices to show that $pH+qH$ is closed and the projection onto $pH+qH$ belongs to $A$. 
Since $A$ is an R$^*$-algebra, the positive operator $p+q$ has closed range by Theorem \ref{equi} (3). 
Since $0\leq p\leq p+q$, we have $p=s(p)\leq s(p+q)$, which implies $pH\subset (p+q)H$. 
Similarly, we have $qH\subset (p+q)H$. 
It is clear that $(p+q)H\subset pH+qH$. 
Thus we obtain $pH+qH=(p+q)H$. 
Recall that $(p+q)H$ is closed and the projection onto $(p+q)H$ is equal to $s(p+q)\in A$. 
This completes the proof of the first item.

To prove the second item, check that the projection onto $pH\cap qH$ is equal to $\chi_{\{1\}}(pqp)$, which belongs to $A$ by Proposition \ref{operation} (2).
The third item is easy.
\end{proof}
This property is very similar to a basic property of the lattice of projections of a von Neumann algebra. 
However, there is a big difference: For projections $p, q$ of a von Neumann algebra in $B(H)$, $p\vee q$ is the projection onto the \emph{closure} of $pH+qH$, and taking closure is inevitable.

In that respect{,} the lattice $\P(A)$ resembles that of operator ranges, namely, the lattice formed of ranges of bounded linear operators on a fixed Hilbert space. 
It is a curious fact that the sum and intersection of any pair of operator ranges are operator ranges. 
However, the lattice of operator ranges is never complemented unless $H$ is finite-dimensional. 
In fact, an operator range has a complement if and only if it is closed. 
See \cite{FW} for the proofs of these facts and a more detailed study of operator ranges.

\begin{remark}
In the above argument{,} we discussed only the case of unital R$^*$-algebras, just for simplicity. 
All of these results carry over to the nonunital case after a suitable change. 
For example, for a nonunital R$^*$-algebra $A${,} the collection of projections forms a modular lattice that is not complemented. 
(To show this, consider the unitization of $A$.)
\end{remark}

Although the following theorem is easy to prove, it gives an enormous motivation to the research on R$^*$-algebras and their lattices of projections.
For a locally compact Hausdorff space $K$, let $C_0(K)$ denote the commutative C$^*$-algebra of continuous complex-valued functions on $K$ that vanish at infinity.

\begin{theorem}\label{boole}
If $A$ is a commutative R$^*$-algebra then the lattice $\P(A)$ forms a generalized Boolean algebra. 
Every generalized Boolean algebra arises in this way. 
Two commutative R$^*$-algebras are $^*$-isomorphic if and only if their projection lattices are isomorphic.
The commutative R$^*$-algebra $A$ is unital if and only if  $\P(A)$ is a Boolean algebra.
\end{theorem}
\begin{proof}
It is a basic fact that the lattice of principal right ideals of a commutative regular ring forms a generalized Boolean algebra. See for example \cite[Theorem II.2.10]{N}, in which unital regular rings are considered.
However, let us give a proof of this fact for the particular case of R$^*$-algebras. 
The proof below visualizes the relation among commutative R$^*$-algebras, Boolean algebras and commutative C$^*$-algebras.

Let $A$ be a commutative R$^*$-algebra. 
Take the completion $\widehat{A}$ of $A$, which is a commutative C$^*$-algebra. 
By the Gelfand--Naimark theorem there exists a locally compact Hausdorff space $K$ such that $\widehat{A}\cong C_0(K)$. 
Then each projection $p\in \P(A)$ can be identified with a projection $f_p$ in $C_0(K)$, which is a continuous function with value in $\{0, 1\}$, and $f_p$ is in turn identified with a compact open subset $f_p^{-1}(1)\subset K$. 
Set $\S:= \{f_p^{-1}(1)\mid p\in \P(A)\}\subset 2^{K}$. 
Then it follows from a quite straightforward argument that the map $p\mapsto f_p^{-1}(1)$ is a lattice isomorphism from $\P(A)$ onto $\S$, which forms a generalized field of sets.

Let $\S\subset 2^X$ be a generalized field of sets. 
Consider the Hilbert space $\ell_2(X)$, which is the space of square-summable complex-valued functions on $X$. 
Then the space of functions $R(\S) := \{f\colon X\to \C\mid \# f(X)<\infty,\,\, f^{-1}(z)\in \S \text{ for every }z\in \C\setminus\{0\}\}$ acts on $\ell_2(X)$ by pointwise multiplication. 
It is clear that $R(\S)$ is a commutative R$^*$-algebra. 
Moreover, for each $f\in \P(R(\S))$, we have $f^{-1}(1)\in \S$, and this gives an identification $\P(R(\S)) = \S$ as generalized Boolean algebras.

Let $A, B$ be commutative R$^*$-algebras. 
It is obvious that if $A$ is $^*$-isomorphic to $B$ then $\P(A)$ is lattice isomorphic to $\P(B)$.
Assume that there is a lattice isomorphism $\psi\colon \P(A)\to \P(B)$. 
We concretely construct a $^*$-isomorphism from $A$ onto $B$. 
Let $x\in A$.
By Theorem \ref{equi},  $x$ has finite spectrum. 
Since $A$ is commutative, $x$ is a normal operator. 
Thus there exist $n\geq 1$, complex numbers $z_1, z_2, \ldots z_n\in \C$, and projections $p_1, p_2, \ldots, p_n\in \P(A)$ such that $x = \sum_{k=1}^n z_kp_k$. 
We define $\phi(x) := \sum_{k=1}^n z_k\psi(p_k)$. 
It is a routine exercise to show that $\phi$ is well-defined (independent of the choice of $z_k$, $p_k$) and is a $^*$-isomorphism from $A$ onto $B$. 

It is clear that  $A$ is unital if and only if $\P(A)$ is a Boolean algebra.
\end{proof}
For a generalized Boolean algebra $\S${,} let $R(\S)$ denote the commutative R$^*$-algebra whose projection lattice is isomorphic to $\S$. 
By the preceding theorem, such an R$^*$-algebra exists uniquely up to $^*$-isomorphism.

The above theorem is highly reminiscent of a similar characterization of commutative C$^*$-/von Neumann algebras. 
The Gelfand--Naimark theorem tells us that the research of  C$^*$-algebras is considered as ``noncommutative locally compact Hausdorff spaces''. 
Similarly, von Neumann algebras can be considered as ``noncommutative measure spaces''. 
The author believes that the theory of R$^*$-algebras can be considered as a theory of ``noncommutative (generalized) Boolean algebras''.

It is known as \emph{Stone's duality theorem} that the category of Boolean algebras is dually equivalent to that of Boolean spaces. 
Recall that a \emph{Boolean space} (also known as a \emph{Stone space}) is a zero-dimensional compact Hausdorff space. 
The lattice of clopen subsets of a Boolean space forms a Boolean algebra, which gives one direction of the correspondence between Boolean spaces and Boolean algebras.

More generally, the category of generalized Boolean algebras is dually equivalent to that of zero-dimensional locally compact Hausdorff spaces. 
A zero-dimensional locally compact Hausdorff space is also called a locally compact Boolean space, and can be characterized as a locally compact Hausdorff space such that every connected component is a singleton.
Let us revisit this correspondence from the viewpoint of R$^*$- and C$^*$-algebras. 
Let $\S$ be a generalized Boolean algebra. 
Then the lattice of projections of the commutative R$^*$-algebra $R(\S)$ is $\S$.
We consider the completion $\widehat{R(\S)}$ of $R(\S)$, which is a commutative C$^*$-algebra. 
\begin{proposition}
Let $\S$ be a generalized Boolean algebra. 
Then every projection in $\widehat{R(\S)}$ is an element of $\P(R(\S))$, which can be identified with $\S$. 
In other words, if one takes a locally compact Hausdorff space $K$ with $\widehat{R(\S)}\cong C_0(K)$ (whose existence is guaranteed by the Gelfand--Naimark theorem), then 
\begin{itemize}
\item $K$ is zero-dimensional, and 
\item the generalized Boolean algebra formed of compact open subsets of $K$ coincides with $\S$. 
\end{itemize}
\end{proposition}
\begin{proof}
Let $p\in \P(\widehat{R(\S)})$. 
Then there exists a sequence $x_n\in R(\S)$ such that $x_n\to p$ (in norm). 
Regard $p$ and $x_n$ as functions in $C_0(K)$. 
Then $p$ is a function with $p(K)\subset \{0, 1\}$, and $x_n\to p$ means that $x_n$ converges to $p$ uniformly. 
Take an open disc $D$ with center $1$ and radius $1/2$ in $\C$. 
It follows that $\chi_{D}(x_n)\to p$. 
Remark that $R(\S)$ is a commutative R$^*$-algebra, so  Proposition \ref{operation} (2) implies that $\chi_{D}(x_n)\in \P(R(\S))$. 
Moreover, $\chi_{D}(x_n)$ can be regarded as a function in $C_0(K)$ with $\chi_{D}(x_n)(K)\subset \{0, 1\}$. 
It follows by $\chi_{D}(x_n)\to p$ that $\chi_{D}(x_n)=p$ for a sufficiently large $n$, which implies $p\in \P(R(\S))$. 
Therefore, the generalized Boolean algebra formed of compact open subsets of $K$ coincides with $\P(R(\S))=\S$. 
Since every function in $C_0(K)$ can be uniformly approximated by elements in $R(\S)$, it is easy to see that each connected component of the locally compact Hausdorff space $K$ is a singleton, which means that $K$ is zero-dimensional.
\end{proof}

Using this, we may interpret Stone's theorem as the equivalence of the following three categories: 
\begin{itemize}
\item the category of generalized Boolean algebras with lattice homomorphisms that preserve $0$, 
\item the category of commutative R$^*$-algebras with $^*$-homomorphisms, 
\item the category of all $C_0(K)$ for zero-dimensional locally compact Hausdorff spaces $K$, with $^*$-homomorphisms.
\end{itemize}
We may give a proof of this equivalence {straightforwardly}. 
The details are omitted, but let us just observe how the morphisms correspond. 
Let $A$, $B$ be commutative R$^*$-algebras. 
Let  $\psi\colon \P(A)\to \P(B)$ be a lattice homomorphism such that $\psi(0)=0$.
The corresponding $^*$-homomorphism $\phi\colon A\to B$ can be described just as in Proof of Theorem \ref{boole}:
Each element $x$ of $A$ can be written as $x=\sum_{k=1}^n z_kp_k$ for some $n\geq 1$, complex numbers $z_1, z_2, \ldots, z_n\in \C$, and $p_1, p_2, \ldots, p_n\in \P(A)$. 
Define $\phi\colon A\to B$ by $\phi(x):= \sum_{k=1}^n z_k\psi(p_k)$, and check that it is a well-defined $^*$-homomorphism.
Secondly, given a $^*$-homomorphism $\phi\colon A\to B$ (which is contractive), we may take its continuous extension $\varphi\colon \widehat{A}\to \widehat{B}$ to completions, which is clearly a $^*$-homomorphism. 
Finally, a $^*$-homomorphism $\varphi\colon \widehat{A}\to \widehat{B}$ maps  $\P(\widehat{A})=\P(A)$ into $\P(\widehat{B})=\P(B)$, thus restricts to a lattice homomorphism that sends $0$ to $0$.

Is it possible to generalize this equivalence to the setting of noncommutative R$^*$-algebras?
Apparently{,} there is no that beautiful correspondence applicable to general R$^*$-algebras.
Indeed,  there exist distinct R$^*$-algebras $A, B$ with $\widehat{A}\cong \widehat{B}$ (see Example \ref{compact} and Section \ref{atom}), so we cannot recover full information about $A$ from $\widehat{A}$.
Moreover, in Proposition \ref{opposite} we prove that there exists a pair of noncommutative R$^*$-algebras which are not $^*$-isomorphic but the lattices of projections are orthoisomorphic. 
Thus the structure of an R$^*$-algebra cannot be fully recovered from the lattice of projections.
Nonetheless, there do exist several nice correspondence results that apply to certain classes of  R$^*$-algebras and their lattices of projections. 
See Sections \ref{countable} and \ref{nd}.

\section{Purely atomic R$^*$-algebras and inner product spaces}\label{atom}
Recall that a von Neumann algebra can be classified into several types by means of projections. 
In particular, a von Neumann algebra is said to be atomic if the least upper bound of minimal projections is the unit.
It is a well-known fact, which follows from e.g.\ the type decomposition theorem of von Neumann algebras, that an atomic von Neumann algebra is an $\ell_\infty$-direct sum of type I factors. 

Let us define a class of R$^*$-algebras that is similar to the class of atomic von Neumann algebras. 
Since it is hard to consider the greatest lower bound of an infinite set of projections in an R$^*$-algebra, we want a definition that involves only finitely many projections.
An \emph{atom} of an R$^*$-algebra $A$ is a minimal projection in $\P(A)\setminus \{0\}$. 
Note that $p\in \P(A)$ is an atom if and only if the R$^*$-algebra $pAp$ is one-dimensional. 
We say an R$^*$-algebra $A$ is \emph{purely atomic} if every projection of $A$ is a sum of finitely many atoms. 
We give a complete description of purely atomic R$^*$-algebras.

\begin{lemma}
Let $A$ be an R$^*$-algebra and $p, q\in \P(A)$ atoms. 
Then the linear space $pAq$ is at most one-dimensional. 
\end{lemma}
\begin{proof}
Let $x, y\in pAq$ be elements with norm one. 
Then $xx^*\in pAp$ is positive and $\lVert xx^*\rVert=\lVert x\rVert^2=1$. 
Since $pAp$ is one-dimensional, we have $xx^*=p$. 
Similarly, we obtain $yy^*=p$ and $x^*x=y^*y=q$. 
It follows that $x, y$ are partial isometries {that} share initial and final spaces. 
Since $xy^*\in pAp$, there is a complex number $\lambda\in \C$ with $\lvert \lambda\rvert=1$ such that $xy^*=\lambda p$. 
It follows that $x=\lambda y$, proving that $x$ and $y$ are linearly dependent.
\end{proof}

\begin{lemma}
Let $A$ be an R$^*$-algebra. 
Then $A$ is purely atomic if and only if $pAp$ is finite-dimensional for every $p\in \P(A)$.
\end{lemma}
\begin{proof}
If $pAp$ is finite-dimensional for every $p\in \P(A)$, it is clear that $A$ is purely atomic. 
Suppose that $A$ is purely atomic and let $p\in\P(A)$. 
Take a collection of atoms $p_1, \ldots, p_n$ with $p=p_1+\cdots+p_n$. 
For each pair $k, l\in \{1, \ldots, n\}$, the preceding lemma implies that $p_kAp_l$ is at most one-dimensional. 
Therefore,  $pAp=\operatorname{span}\{p_kAp_l\mid k, l\in \{1, \ldots, n\}\}$ is finite-dimensional.
\end{proof}

It follows that a purely atomic R$^*$-algebra is ultramatricial. 
Indeed, for a purely atomic R$^*$-algebra $A$ and a collection $x_1, \ldots, x_n\in A$, define $p:=\bigvee_{1\leq i\leq n}{l(x_i)}\vee\bigvee_{1\leq i\leq n}{r(x_i)}$, then $pAp$ is finite-dimensional and $x_1, \ldots, x_n\in pAp$. 

Let $V$ be an inner product space. 
Let $H$ be the completion of $V$. 
Consider the collection $F(V)$ of operators $x$ in $B(H)$ with finite rank such that $r(x)H, l(x)H\subset V$ (namely, such that the ranges of $x$ and $x^*$ are included in $V$). 
It is easy to see that $F(V)$ is a purely atomic R$^*$-algebra. 

\begin{theorem}\label{vi}
Let $(V_i)_{i\in I}$ be a family of complex inner product spaces.
Then the R$^*$-algebra $\bigoplus_{i\in I} F(V_i)$ is purely atomic.
Every purely atomic R$^*$-algebra is $^*$-isomorphic to some R$^*$-algebra that can be obtained in this manner. 
\end{theorem}
\begin{proof}
It is clear that the space  $\bigoplus_{i\in I} F(V_i)$ as in the statement of this theorem forms a purely atomic R$^*$-algebra. 
We prove that every purely atomic R$^*$-algebra is of this form.
The following proof is based on the theory of von Neumann algebras.

Let $A$ be a purely atomic R$^*$-algebra. 
Realize $A$ as a $^*$-subalgebra of some $B(H)$. 
Take the $\sigma$-weak closure (or equivalently, the  $\sigma$-strong closure) $M$ of $A$ in $B(H)$. 
Consider $\P(A)$ as an increasing net of projections in $M$. 
The fact that for every $x\in A$ the projection $p=l(x)\vee r(x)\in \P(A)$ satisfies $pxp=x$ implies that $\bigvee_{p\in \P(A)}p\in \P(M)$ is equal to the unit of $M$. 
Moreover, for each $p\in \P(A)$, the finite-dimensionality of $pAp$ implies that $pMp=pAp$. 
It follows that $M$ is an atomic von Neumann algebra. 
Thus we may find a collection $(H_i)_{i\in I}$ of Hilbert spaces such that $M$ is $^*$-isomorphic to the $\ell_{\infty}$-direct sum of $(B(H_i))_{i\in I}$. 
Since $pAp=pMp$ is finite-dimensional for each $p\in\P(A)$, each operator in $A$ can be thought of as an operator of finite rank in the $\ell_{\infty}$-direct sum of $(B(H_i))_{i\in I}$. 
Therefore, we may and do identify $A$ with a $^*$-subalgebra of the R$^*$-algebra $\bigoplus_{i\in I}F(H_i)$ ($\subset B(K)$, where $K:=\bigoplus_{i\in I} H_i$). 

Let us check that there exists a dense subspace $V_i\subset H_i$ for each $i\in I$ such that $A$ is identified with $\bigoplus_{i\in I}F(V_i)\,\,(\subset \bigoplus_{i\in I}F(H_i))$.
Fix $i\in I$. 
Define $V_i:= \{h\in H_i \mid h\in \operatorname{ran}p\text{ for some }p\in \P(A)\}$. 
Let $p_i$ be the projection from $K$ onto $H_i$, which lies in the center of $M$.
Observe that for each $p\in \P(A)$ we have $pp_i\in pMp=pAp$. 
Thus we see that $V_i$ is equal to $\{h\in H_i \mid h\in \operatorname{ran}p\text{ for some }p\in \P(A) \text{ with }p\leq p_i\}$.
By the fact that $\P(A)$ forms a lattice, we see that $V_i$ is actually a linear subspace of $H_i$.
Since $A$ is $\sigma$-strongly dense in $M$, we also see that $V_i$ is dense in $H_i$. 
For each $x\in A$, there exist $i_1, \ldots, i_n\in I$ such that $x=\sum_{k=1}^n p_{i_k}x$. 
Set $q_k=p_{i_k}(l(x)\vee r(x))$, $1\leq k\leq n$. 
Then we have $q_k\in \P(A)$ and $q_k\leq p_{i_k}$, thus $p_{i_k}x= q_kxq_k\in F(V_{i_k})$. 
This shows that $A\subset \bigoplus_{i\in I}F(V_i)$.
Let $i\in I$ and $p\in F(V_i)\,\,(\subset F(H_i))$ be a projection of rank one.  
The definition of $V_i$ implies that there is a projection $q\in \P(A)$ with $p\leq q$. 
Thus we have $p\in qMq$, but $q\in \P(A)$ implies that $p\in qMq=qAq\subset A$. 
Since the collection of rank-one projections (atoms) in $F(V_i)$ linearly spans $F(V_i)$, we obtain $F(V_i)\subset A$. 
This implies $\bigoplus_{i\in I}F(V_i)\subset A$, and the proof is complete.
\end{proof}

\begin{corollary}
A purely atomic simple R$^*$-algebra is $^*$-isomorphic to $F(V)$ for some inner product space $V$. 
\end{corollary}

Therefore, the classification of purely atomic R$^*$-algebras reduces to that of R$^*$-algebras of the form $\bigoplus_{i\in I} F(V_i)$.
\begin{theorem}\label{wigner}
Let $(V_i)_{i\in I}$, $(W_j)_{j\in J}$ be two families of complex inner product spaces.
Let $H_i$, $K_j$ be the completions of $V_i$, $W_j$, respectively, for $i\in I$, $j\in J$.
Let $\phi\colon \bigoplus_{i\in I} F(V_i)\to \bigoplus_{j\in J} F(W_j)$ be a $^*$-isomorphism. 
Then there exist a bijection $\tau\colon I\to J$ and a linear surjective isometry $u_i\colon H_i\to K_{\tau(i)}$ for each $i\in I$ such that $u_i(V_i)=W_{\tau(i)}$ and $\phi(x) = u_ixu_i^{-1}\in F(W_{\tau(i)})\,\,(\subset \bigoplus_{j\in J} F(W_j))$ for every $x\in F(V_i)$, $i\in I$. 
\end{theorem}
\begin{proof}
Since $\phi$ is a $^*$-isomorphism between two R$^*$-algebras, it is isometric and preserves the collection of atoms. 
Thus $\phi$ maps each connected component of the set of atoms in $\bigoplus_{i\in I} F(V_i)$ onto some connected component of the set of atoms in $\bigoplus_{j\in J} F(W_j)$.
Fix $i_0\in I$. 
Then the set $S$ of atoms in $F(V_{i_0})$ is a connected component of the set of atoms in $\bigoplus_{i\in I} F(V_i)$. 
There exists a unique $\tau(i_0)=j_0\in J$ such that $\phi(S)$ coincides with the collection of atoms in $F(W_{j_0})$.
Set $V:=V_{i_0}$, $H:=H_{i_0}$, $W:=W_{j_0}$ and $K:=K_{j_0}$.
Assume that $V$ is at least two-dimensional. 
Observe that $S$ can be identified with a dense subset of the collection of atoms in $F(H)$.
Since a $^*$-isomorphism is isometric, we may apply Wigner's theorem (see for example \cite{Geh}) to show that there exists a linear or conjugate-linear surjective isometry $u\colon H\to K$ such that $u(V)=W$ and $\phi(p) =upu^{-1}$ for every $p\in S$. 
Take a pair of self-adjoint operators $a, b\in F(V)$ with $ab\neq ba$.
Since $\phi$ is linear and every self-adjoint operator in $F(V)$ can be written as a finite real-linear combination of $S$, we have 
\[
\phi(ab) = \phi\left(\frac{ab+ba}{2} + i\frac{ab-ba}{2i}\right) = u\frac{ab + ba}{2}u^{-1} + iu\frac{ab-ba}{2i}u^{-1}.
\]
This is equal to $uabu^{-1}$ if $u$ is linear, and to $ubau^{-1}$ if $u$ is conjugate-linear. 
On the other hand, since $\phi$ is multiplicative, we have 
\[
\phi(ab) = \phi(a)\phi(b) = uau^{-1}ubu^{-1}=uabu^{-1}.
\]
Since $ab\neq ba$, we conclude that $u$ is actually linear. 
It follows that $\phi(x) =uxu^{-1}$ for every $x\in F(V)$. 
If $V=H$ is one-dimensional then so is $W=K$. 
In this case it is clear that for any linear isometry $u\colon H\to K$ we have $\phi(x) =uxu^{-1}$ for every $x\in F(V)$. 
It is easy to see that $\tau\colon I\to J$ constructed in the above manner is a bijection.
\end{proof}

In particular, the classification of purely atomic simple R$^*$-algebras is equivalent to that of inner product spaces. 
How rich is the class of inner product spaces? 
If we restrict ourselves to inner product spaces with countable Hamel dimension, then there is nothing interesting: 
Every inner product space with countably infinite Hamel dimension is linearly isometric to $c_{00}$ as a subspace of $\ell_2$. (Proof: Consider the Gram--Schmidt procedure.)
On the other hand, we know that complete inner product spaces ($=$ Hilbert spaces) are completely classified by the cardinality of orthonormal basis. 

In contrast, the behavior of incomplete inner product spaces with uncountable Hamel dimension is mysterious (even in the separable case). 
A class of inner product spaces with uncountable Hamel dimension is given by ranges of operators: 
If  $X$ is a  Banach space and  $T\colon X\to H$ is a bounded linear operator, then the range $TX$ of $T$ is a linear subspace of $H$, which forms an inner product space.

\begin{proposition}\label{range}
\begin{enumerate}
\item For every separable Banach space $X$, there exists an injective bounded linear operator $T\colon X\to \ell_2$. 
\item Let $X_i$ be a Banach space, $H_i$ be a Hilbert space, and $T_i\colon X_i\to H_i$ be an injective bounded linear operator with image $T_iX_i=V_i\subset H_i$, $i=1, 2$. 
If the R$^*$-algebras $F(V_1)$ and $F(V_2)$ are $^*$-isomorphic, then $X_1$ and $X_2$ are isomorphic as Banach spaces, i.e., there exists a bounded linear bijection between $X_1$ and $X_2$. 
\end{enumerate}
\end{proposition}
\begin{proof}
This proposition is a consequence of well-known facts (see for example \cite[Proposition 3.1 and Remark 3.14]{COS}). 
For convenience{,} we give a proof.\\
(1) Take a sequence $x_n$, $n\geq 1$, of elements that is dense in the unit sphere of $X$. 
By the Hahn--Banach theorem for each $n\geq 1$ we may take a linear functional $f_n\colon X\to \C$ with $\lVert f_n\rVert =1$, $f_n(x_n)=1$. 
Define the map $T\colon X\to \ell_2$ by $T(x) = (f_n(x)/2^n)_{n\geq 1}$. 
It is a routine exercise to show that  $T$ is well-defined, bounded, linear, and injective.\\
(2) Take a $^*$-isomorphism $\phi\colon F(V_1)\to F(V_2)$. 
By Theorem \ref{wigner} we may take a linear surjective isometry $u\colon K_1\to K_2$ such that $u(V_1)=V_2$ and $\phi(x) = uxu^{-1}$, $x\in F(V_1)$, where $K_i$ is the completion of $V_i$, $i=1, 2$. 
It is easy to see that the mapping $T_2^{-1}\circ u\circ T_1\colon X_1\to X_2$ is a linear bijection with closed graph. 
It follows by the closed graph theorem that this mapping is a continuous bijection (with continuous inverse).
\end{proof}

Since there are many isomorphism classes of separable Banach spaces, we see that there are many simple purely infinite R$^*$-algebras arising from ranges of operators. 
What if we restrict ourselves to ranges of operators in $B(H)$?
\begin{proposition}
Let $x, y\in B(H)$ be injective operators with dense range. 
Then the R$^*$-algebras $F(\operatorname{ran} x)$ and $F(\operatorname{ran} y)$ are $^*$-isomorphic if and only if there exist bounded linear bijections $a, b\in B(H)$ such that $y=axb$. 
\end{proposition}
\begin{proof}
This is a corollary to \cite[Theorem 3.4]{FW}. 
For convenience we repeat a proof.
Suppose that $F(\operatorname{ran} x)$ and $F(\operatorname{ran} y)$ are $^*$-isomorphic. 
By Theorem \ref{wigner} there exists a unitary operator $u\colon H\to H$ with $u(\operatorname{ran} x)=\operatorname{ran}y$. 
It follows by the closed graph theorem that the map $a= y^{-1}\circ u\circ x$ is a bounded linear bijection on $H$, and $y=uxa^{-1}$.
Let $a, b\in B(H)$ be bounded linear bijections and $y=axb$. 
Then $\operatorname{ran} y = \operatorname{ran}(ax) = \operatorname{ran} (a(xx^*)^{1/2})$. 
Let $a(xx^*)^{1/2}=v\lvert a(xx^*)^{1/2}\rvert$ be the polar decomposition of $a(xx^*)^{1/2}$. 
Then $v$ is unitary and 
\[
\operatorname{ran} y 
= \operatorname{ran} (v\lvert a(xx^*)^{1/2}\rvert)
= v(\operatorname{ran} (\lvert a(xx^*)^{1/2}\rvert))= v(\operatorname{ran} ((xx^*)^{1/2}a^*) =v(\operatorname{ran} x),
\]
which implies that $F(\operatorname{ran} x)$ is $^*$-isomorphic to $F(\operatorname{ran} y)$.
\end{proof}

In the above proposition, we see for example that if in addition $x$ belongs to the Schatten--von Neumann $p$-ideal for some $1\leq p<\infty$, then so does $y$.
Therefore, there are many examples of mutually nonisometric inner product spaces that arise as ranges of compact operators in $B(H)$. 
See \cite{FW}  and \cite{COS} for more on operator ranges.

Let us show that purely atomic R$^*$-algebras can also be algebraically characterized by {an} Artinian-type condition. 
\begin{proposition}
An R$^*$-algebra $A$ is purely atomic if and only if there exists no strictly decreasing infinite sequence of projections $p_1\geq p_2\geq\cdots$ in $A$, or in other words, there exists no infinite chain of descending principal right ideals of $A$. 
\end{proposition}
\begin{proof}
Every purely atomic R$^*$-algebra $A$ satisfies the latter condition because for every $p\in \P(A)$ the algebra $pAp$ is finite-dimensional. 
Suppose that $A$ is not purely atomic. 
Consider the collection $S$ of all projections $p$ with the following property, which is nonempty by the assumption: 
For any pairwise orthogonal subprojections $q_1, q_2, \ldots, q_n$ of $p$ with $p=q_1+q_2+\cdots+q_n$, at least one of  $q_1, q_2, \ldots, q_n$ is not an atom. 
If $p\in S$, then $p$ is not an atom, so there exists a subprojection $q$ of $p$ such that $q\neq 0\neq p-q$. 
It is easy to see that at least one of the two projections $q, p-q$ is in $S$. 
Therefore, we may choose a strictly decreasing infinite sequence of projections $p_1\geq p_2\geq\cdots$ in $S$. 
\end{proof}

What if one considers ascending chain condition instead of descending one?
\begin{proposition}\label{descend}
For an R$^*$-algebra $A$ the following three conditions are equivalent. 
\begin{itemize}
\item There exists no strictly increasing infinite sequence of projections $p_1\leq p_2\leq\cdots$ in $A$, or in other words, there exists no infinite chain of ascending principal right  ideals of $A$.
\item There is no infinite family of mutually orthogonal nonzero projections in $A$.
\item $A$ is finite-dimensional.
\end{itemize}
\end{proposition}
\begin{proof}
It is clear that the first and the second conditions are equivalent and led by the third condition.
Let $A$ satisfy the first condition. 
Then clearly $\P(A)$ has a maximal element. 
Since $\P(A)$ is a lattice it is the unique maximal element of $\P(A)$. 
Moreover{,} it is easy to see that the unique maximal element of $\P(A)$ must be a unit of $A$. 
It follows that there exists no strictly decreasing infinite sequence of projections $p_1\geq p_2\geq\cdots$ in $A$, because if there were then the sequence $1-p_1,  1-p_2, \ldots$ would be strictly increasing. 
Therefore, $A$ is purely atomic and unital, which together with Theorem \ref{vi} implies the finite-dimensionality of $A$.
\end{proof}

\section{R$^*$-algebras in comparison with von Neumann algebras, Baer $^*$-rings and continuous geometries}\label{comparison}
Again let us recall the theory of von Neumann algebras. 
In the type classification of von Neumann algebras, two important concepts appear: abelian projections and finite projections. 
Recall that a projection $p$ of a von Neumann algebra $M$ is finite if every partial isometry $v\in M$ with $vv^*\leq v^*v=p$ satisfies $vv^*=p$. 
Finiteness of projections in R$^*$-algebras is automatic. 
\begin{proposition}\label{finite}
Let $A$ be an R$^*$-algebra and $v\in A$ be a partial isometry. 
If $vv^*\leq v^*v$, then $vv^*=v^*v$.
\end{proposition}
\begin{proof}
Assume that $v\in A$ is a partial isometry with $vv^*\leq v^*v$ and $vv^*\neq v^*v$. 
Then the operator $v$ restricted to $(\ker v)^{\perp}$ is a proper isometry into $(\ker v)^{\perp}$. 
It is well-known that a proper isometry has infinite spectrum.
(To prove it, use e.g.\ the Wold--von Neumann decomposition theorem and the fact that the unilateral shift has infinite spectrum.)
This contradicts the assumption that $A$ is an R$^*$-algebra.
\end{proof}

In what follows we consider abelian projections.
For a von Neumann algebra $M$, recall that $M$ is of type I if and only if there is a family of abelian projections whose least upper bound is the unit.

We say a projection $p$ of an R$^*$-algebra $A$ is \emph{abelian} if the R$^*$-algebra $pAp$ is commutative. 
Clearly, an atom is an abelian projection.
The author is not aware whether there is a natural class of R$^*$-algebras that plays a role similar to that played by the class of type I von Neumann algebras in the theory of von Neumann algebras.

Let $A$ be an R$^*$-algebra. 
Two projections $p, q\in \P(A)$ are said to be \emph{equivalent} in $A$ (in the sense of Murray--von Neumann, and written $p\sim q$)  if there exists a partial isometry $v\in A$ with $v^*v=p$ and $vv^*=q$.
This is an equivalence relation in $\P(A)$. 

\begin{lemma}\label{lr}
Let $A$ be an R$^*$-algebra. 
For $x\in A$, $l(x)\sim r(x)$.
\end{lemma}
\begin{proof}
Let $x=v\lvert x\rvert$ be the polar decomposition. 
By Proposition \ref{operation} (3), $v$ is a partial isometry in $A$. 
Since $vv^*=l(x)$ and $v^*v=r(x)$, we have $l(x)\sim r(x)$.
\end{proof}

\begin{lemma}\label{p1p2}
Let $A$ be an R$^*$-algebra and $p\in \P(A)$. 
Then $p$ is nonabelian if and only if there exists a pair of nonzero subprojections $p_1, p_2$ of $p$ in $A$ with $p_1p_2=0$ and $p_1\sim p_2$. 
\end{lemma}
\begin{proof}
If the latter condition is satisfied then the algebra $pAp$ has a partial isometry $v$ with $v^*v=p_1\neq p_2=vv^*$, thus $p$ is nonabelian. 
Suppose that $p$ is nonabelian, which means that $pAp$ is noncommutative. 
Since $pAp$ is an R$^*$-algebra, it is spanned by projections in $pAp$. 
It follows that there exists a pair of projections $p_1, p_2\in \P(pAp)$ such that $p_1p_2\neq p_2p_1$. 
Thus $p_1p_2-p_2p_1p_2$ is not self-adjoint, and we obtain  $x:=(p-p_2)p_1p_2 = p_1p_2-p_2p_1p_2\neq 0$.
Then $l(x)\leq p-p_2$ and $r(x)\leq p_2$, thus $l(x)$ and $r(x)$ are mutually orthogonal and equivalent.
\end{proof}

The \emph{center} of an R$^*$-algebra $A$ is the collection $\{x\in A\mid xy=yx\text{ for all }y\in A\}$. 
An element of $A$ is said to be central if it belongs to the center of $A$.

\begin{lemma}\label{central}
Let $A$ be an R$^*$-algebra and $p\in\P(A)$. 
Then $p$ is noncentral if and only if there exists a pair of nonzero projections $p_1, p_2\in \P(A)$ such that $p_1\leq p$, $p_2p=0$ and $p_1\sim p_2$. 
\end{lemma}
\begin{proof}
If $p$ is central then for every partial isometry $v$ with $v^*v\leq p$ we have $vv^*=vpv^*=pvv^*p\leq p$, so there is no pair $p_1, p_2$ with the property. 
Assume that $p$ is not central. 
Since $A$ is spanned by projections there is a projection $q\in A$ such that $pq\neq qp$.
It follows that $x:=pq(p\vee q-p)=pq-pqp$ is not self-adjoint and thus nonzero. 
Then $p_1:=l(x)\leq p$ and $p_2:=r(x)\leq p\vee q-p$ satisfy the desired property.
\end{proof}

Recall that a ring $R$ is \emph{Baer} if the right annihilator of every subset of $R$ is generated by an idempotent as a principal right ideal, that is, if for every $X\subset R$ there exists an idempotent $e\in R$ (i.e., $e=e^2$) such that $\{r\in R\mid Xr=0\} = eR$.
A von Neumann algebra is a Baer $^*$-ring.
The structure of Baer $^*$-rings is studied by Kaplansky \cite{Kap} (see also \cite{Be}) as a purely algebraic generalization of von Neumann algebras. 
Which R$^*$-algebras are Baer? 
Below we give the answer (Proposition \ref{baer}). 
Recall that a lattice $\P$ is said to be \emph{complete} if any family $(p_i)_{i\in I}$ of elements in $\P$ has the least upper bound $\bigvee_{i\in I} p_i$ (upward complete) and the greatest lower bound $\bigwedge_{i\in I} p_i$ (downward complete).
First let us recall a fact, which is actually true for general $^*$-regular rings \cite[Proposition 4.1]{Be}, and sketch the proof. 

\begin{lemma}\label{completeness}
For an R$^*$-algebra $A$ the following conditions are equivalent.
\begin{enumerate}
\item $A$ is Baer. 
\item $\P(A)$ is complete.
\item $\P(A)$ is upward complete.
\end{enumerate}
\end{lemma}
\begin{proof}
Recall that every principal right ideal of a $^*$-regular ring is generated by a projection. \\
$(1)\Rightarrow(3)$
Assume that $A$ is Baer. 
It clearly follows that $A$ is unital. 
Let $(p_i)_{i\in I}$ be a family of elements in $\P(A)$. 
Let the right annihilator of $(p_i)_{i\in I}$ be equal to $qA$ with $q\in \P(A)$. 
Then we may check that $1-q$ is the least upper bound of $(p_i)_{i\in I}$, hence $\P(A)$ is upward complete.\\
$(3)\Rightarrow(2)$
If $\P(A)$ is upward complete then the least upper bound of $\P(A)$ needs to be a unit of $A$, hence $A$ is unital.
Then $\P(A)$ is downward complete, too.
Indeed, for each family $(p_i)_{i\in I}$ of projections the greatest lower bound is determined by $1- \bigvee_{i\in I} (1-p_i)$.\\
$(2)\Rightarrow(1)$ For a subset $X\subset R$, the right annihilator of $X$ is the principal right ideal generated by $\bigwedge_{x\in X}(1-r(x))$.
\end{proof}

It is also common to consider completeness with respect to countably infinite collections. 
A lattice $\P$ is said to be \emph{$\sigma$-complete}  if any countable family $(p_i)_{i\in I}$ of $\P$ has the least upper bound $\bigvee_{i\in I} p_i$ (upward $\sigma$-complete) and the greatest lower bound $\bigwedge_{i\in I} p_i$ (downward $\sigma$-complete).
Recall that these concepts are important in measure theory. 
A $\sigma$-field is a $\sigma$-complete Boolean algebra, and for a measure space $(X, \mathcal{F}, \mu)$ the collection $\{A\in \mathcal{F} \mid \mu(A)<\infty\}$ forms a downward $\sigma$-complete generalized Boolean algebra.
Notice that a $\sigma$-complete generalized Boolean algebra need not be Boolean. 
For example, for an uncountable set $X$, consider the generalized field of sets formed of all (at most) countable subsets of $X$. 
It should be also mentioned that a $\sigma$-complete Boolean algebra need not be orthoisomorphic to a $\sigma$-field (see \cite[Chapter 40]{GH}). 
Let us characterize R$^*$-algebras whose lattice of projections are upward/downward $\sigma$-complete. 

\begin{proposition}\label{sigma}
For an R$^*$-algebra $A$ the following conditions are equivalent.
\begin{enumerate}
\item $\P(A)$ is $\sigma$-complete.
\item $\P(A)$ is upward $\sigma$-complete.
\item There exist a $\sigma$-complete generalized Boolean algebra $\S$ and a finite-dimensional R$^*$-algebra $B$ such that $A$ is $^*$-isomorphic to $R(\S)\oplus B$.
\end{enumerate}
\end{proposition}
\begin{proof}
$(3)\Rightarrow(1)$ Easy. $(1)\Rightarrow(2)$ This is trivial.\\ 
$(2)\Rightarrow(3)$ Assume $\P(A)$ is upward $\sigma$-complete. 
First we claim that there exists no sequence $(p_n)_{n\geq 1}$ of pairwise orthogonal nonzero nonabelian projections in $A$. 
Indeed, if there exists such a sequence, then Lemma \ref{p1p2} implies that for each $n\geq 1$ there exist mutually equivalent nonzero subprojections $q_n, r_n\leq p_n$ with $q_nr_n=0$. 
Take a partial isometry $v_n\in A$ with $v_n^*v_n=q_n, v_nv_n^*=r_n$. 
Note that $\{q_n, v_n, v_n^*, r_n\}$ forms a system of $2\times 2$ matrix units.
Set $q=\bigvee_{n\geq 1}q_n$ and 
\[
r=\bigvee_{n\geq 1} \left((1-n^{-1})q_n + \sqrt{(1-n^{-1})n^{-1}}(v_n+v_n^*) + n^{-1}r_n\right)
\]
Then we have $p_nqp_n = q_n$ and 
\[
p_nrp_n=(1-n^{-1})q_n + \sqrt{(1-n^{-1})n^{-1}}(v_n+v_n^*) + n^{-1}r_n
\]
for each $n\geq 1$. 
From this{,} it is not difficult to see that the sum $q+r$ has infinite spectrum, which contradicts Theorem \ref{equi}.
Thus we get the desired claim.

Consider the collection $S$ of all projections in $\P(A)$ that cannot be written as a sum of finitely many mutually orthogonal abelian projections. 
We prove that for each $p\in S$ there exists a subprojection $e$ of $p$ such that $e\in S$ and $p-e$ is nonabelian.
Let $p\in S$. 
Since $p$ is clearly nonabelian, there is a pair of subprojections $p', p''\in \P(A)$ of $p$ such that $p'p''=0$ and $p'\sim p''$. 
If $p'$ is not abelian, then it is clear that both $p'$ and $p-p'$ are nonabelian and at least one of $p', p-p'$ is in $S$, which leads to the desired conclusion.
If $p'$ is abelian, then it is clear that $p'+p''$ is nonabelian and does not belong to $S$, and hence $p-(p'+p'')$ is in $S$, which also leads to the desired conclusion.
Therefore, if $S\neq \emptyset$, then we may inductively take a sequence of mutually orthogonal nonzero nonabelian projections in $A$, which contradicts the above claim. 
It follows that $S=\emptyset$.
 
Let $p\in \P(A)$ be an abelian projection.
We prove that  there exists a subprojection $f(p)\leq p$ such that 
\begin{enumerate}
\item[(a)] there exists a collection $e_1, e_2, \ldots, e_n$ of mutually orthogonal atoms in $A$, none of which is central in $A$, such that $f(p)=e_1+\cdots+e_n$, and  
\item[(b)] $p-f(p)$ is a central abelian projection in $A$.
\end{enumerate}
To prove it, take a maximal family $(p_i)_{i\in I}$ of pairwise orthogonal nonzero subprojections of $p$ with the following property: 
For each $i\in I$ there exists a projection $q_i\in \P(A)$ such that $p_iq_i=0$ and $q_i\sim p_i$
(the existence of a maximal family with this property is ensured by Zorn's lemma). 
Then $(q_i)_{i\in I}$ is pairwise orthogonal. 
Indeed, if $q_iq_j\neq 0$ for some $i\neq j\in I$ then $l(q_iq_j)\,\, (\leq q_i)$ is equivalent to $r(q_iq_j)\,\, (\leq q_j)$, which must imply that some nonzero subprojection of $p_i$ is equivalent to a subprojection of $p_j$, and this contradicts the commutativity of $pAp$.
Similarly, we have $p_iq_j=0$ if $i\neq j\in I$.
It follows that $(p_i+q_i)_{i\in I}$ is a family of pairwise orthogonal nonzero nonabelian projections in $A$, which forces $I$ to be a finite set. 
Set $f(p):= \sum_{i\in I}p_i$.
The above reasoning also shows that there is $q\in \P(A)$ with $f(p)\sim q$ and $f(p)q=0$.
Take a partial isometry $v$ such that $v^*v=f(p)$ and $vv^*=q$.
Assume that $f(p)Af(p)$ is infinite-dimensional. 
Proposition \ref{descend} implies that there is a sequence of pairwise orthogonal nonzero subprojections $(p_n)_{n\geq 1}$ of $f(p)$. 
It follows that the sequence $(p_n+ vp_nv^*)_{n\geq 1}$ is a pairwise orthogonal sequence of nonzero nonabelian projections, which contradicts the above claim. 
Therefore, we see that $f(p)Af(p)$ is finite-dimensional.
It is now clear that $f(p)$ is a sum of finitely many pairwise orthogonal atoms none of which is central in $A$.
By Lemma \ref{central} combined with the maximality of the family $(p_i)_{i\in I}$, $p-f(p)$ is central in $A$. 

Let $p\in \P(A)$. 
Then $p$ is the sum of some finitely many pairwise orthogonal abelian projections. 
It follows from the preceding paragraph that there is a subprojection $f(p)$ of $p$ with (a) and (b).
A moment's reflection shows that such an $f(p)$ is unique and that $A$ can be decomposed into the direct sum of two unions of increasing nets of R$^*$-algebras $(f(p)Af(p))_{p\in \P(A)}$ and $((p-f(p))A(p-f(p)))_{p\in \P(A)}$. 
Clearly, the directed union of the family $(f(p)Af(p))_{p\in \P(A)}$ is purely atomic and that of $((p-f(p))A(p-f(p)))_{p\in \P(A)}$ is abelian.
It follows that there exist a $\sigma$-complete generalized Boolean algebra $\S$ and a purely atomic R$^*$-algebra $B$, none of whose atoms is central, such that $A$ is $^*$-isomorphic to $R(\S)\oplus B$.

Lastly, let us prove that $B$ is actually finite-dimensional. 
By Theorem \ref{vi} $B$ is $^*$-isomorphic to $\bigoplus_{j\in J}F(V_j)$ for some family of inner product spaces $(V_j)_{j\in J}$. 
Since none of the atoms of $B$ is abelian, none of $(V_j)_{j\in J}$ is one-dimensional. 
Assume that $B$ is infinite-dimensional. 
Then it is an easy exercise to find an infinite family of pairwise orthogonal nonzero nonabelian projections in $B$, which contradicts the above claim. 
This means that $B$ is finite-dimensional, which completes the proof. 
\end{proof}

\begin{proposition}
For an R$^*$-algebra $A$, $\P(A)$ is downward $\sigma$-complete if and only if there exist a downward $\sigma$-complete generalized Boolean algebra $\S$ and a purely atomic R$^*$-algebra $B$ such that $A$ is $^*$-isomorphic to $R(\S)\oplus B$.
\end{proposition}
\begin{proof}
Assume that $\P(A)$ is downward $\sigma$-complete.
Let $p\in \P(A)$. 
Then $pAp$ is upward $\sigma$-complete. 
Indeed, for a countable family of projections $(p_i)_{i\in I}$ in $pAp$, its least upper bound is equal to $p-\bigwedge_{i\in I}(p-p_i)$. 
Thus we may apply the proof of the preceding theorem verbatim to show that $A$ can be decomposed into the direct sum of two unions of increasing nets of R$^*$-algebras $(f(p)Af(p))_{p\in \P(A)}$ and $((p-f(p))A(p-f(p)))_{p\in \P(A)}$, where $f(p)\leq p$ satisfies (a) and (b). 
It follows that there exist a downward $\sigma$-complete generalized Boolean algebra $\S$ and a purely atomic R$^*$-algebra $B$ (none of whose atom is central) such that $A$ is $^*$-isomorphic to $R(\S)\oplus B$.
It is easy to prove the converse implication.
\end{proof}

The characterization of Baer R$^*$-algebras is a special case of Proposition \ref{sigma}.

\begin{proposition}\label{baer}
Let $A$ be an R$^*$-algebra.
Then $A$ is Baer if and only if there exist a complete Boolean algebra $\S$ and a finite-dimensional R$^*$-algebra $B$ such that $A$ is $^*$-isomorphic to $R(\S)\oplus B$.
\end{proposition}
\begin{proof}
Use Lemma \ref{completeness} and apply  Proposition \ref{sigma}. 
(Beware of the fact that a Baer ring is unital.)
\end{proof}

\begin{proposition}
For an R$^*$-algebra $A$, $\P(A)$ is downward complete if and only if there exist a downward complete generalized Boolean algebra $\S$ and a purely atomic R$^*$-algebra $B$ such that $A$ is $^*$-isomorphic to $R(\S)\oplus B$.
\end{proposition}

Recall that a complete complemented modular lattice $\P$ is called a (possibly reducible) \emph{continuous geometry} if it satisfies 
\[
\left(\bigvee_{i\in I} p_i\right) \wedge q = \bigvee_{i\in I} (p_i \wedge q)
\]
and 
\[
\left(\bigwedge_{i\in I} p_i\right) \vee q = \bigwedge_{i\in I} (p_i \vee q)
\]
for any linearly ordered family $(p_i)_{i\in I}$ of elements in $\P$ and $q\in \P$. 
Von Neumann studied continuous geometries in \cite[Parts I and III]{N}. 
Kaplansky \cite{Kap2} proved that every orthocomplemented complete modular lattice is actually a continuous geometry. 
Thus we obtain:

\begin{corollary}
Let $A$ be an R$^*$-algebra. 
The projection lattice $\P(A)$ is a continuous geometry if and only if the equivalent conditions of Proposition \ref{baer} hold.
\end{corollary}
We may prove it without resorting to Kaplansky's theorem.
\begin{proof}
If $\P(A)$ is a continuous geometry, then $\P(A)$ is complete, so $A$ is of the form of Proposition \ref{baer}. 
It is easy to check that the projection lattice of an R$^*$-algebra of this form is a continuous geometry.
\end{proof}

Therefore, R$^*$-algebras themselves do not provide interesting examples of continuous geometries. 
However, it should be worth mentioning that two important examples of continuous geometries can be constructed from sequentially ultramatricial R$^*$-algebras. 

For a von Neumann algebra $M$ of finite type, the projection lattice of $M$ forms a continuous geometry. 
(It is a remarkable fact that $\P(M)$ is identified with the projection lattice of the $^*$-regular algebra $\mathcal{U}(M)$ formed of all affiliated operators relative to $M$. 
Observe that for the algebra $A=\mathcal{U}(M)$, which can be considered as an algebra of unbounded operators on a Hilbert space, the equivalence of Theorem \ref{equi} is far from true.)
In particular, the projection lattice of the AFD II$_1$ factor $\cR$ is a continuous geometry. 
Recall that $\cR$ arises as the $\sigma$-weak closure of the R$^*$-algebra $\bigotimes_{k\geq 1}\M_{n_k}$ (represented on a Hilbert space through the GNS-construction with respect to the trace), and it does not depend on the choice of integers $n_k\geq 2$. 

On the other hand, von Neumann also considered the completion of $\bigotimes_{k\geq 1}\M_{n_k}$ with respect to a metric called the rank metric, and proved that the lattice of projections forms a continuous geometry that does not depend on the choice of $(n_k)_{k\geq 1}$. 
He also proved that the lattice obtained in this way is different from any projection lattice of a finite von Neumann algebra. 
See \cite{N2} for the details.

\section{Classification of sequentially ultramatricial R$^*$-algebras}\label{countable}
In this section, we give a classification result on locally finite-dimensional R$^*$-algebras with countable Hamel dimension. 
Clearly, an R$^*$-algebra has a countable Hamel basis if and only if it is countably generated as a $^*$-algebra. 
It follows that the class of locally finite-dimensional R$^*$-algebras with countable Hamel dimension coincides with that of sequentially ultramatricial R$^*$-algebras.

It was shown by Bratteli \cite{Br} (after Glimm's classification of UHF algebras \cite{Gli}) that the classification problem of separable AF C$^*$-algebras is equivalent to that of sequentially ultramatricial R$^*$-algebras. 
The classification of sequentially ultramatricial R$^*$-algebras was completed by Elliott in the seminal paper \cite{E}.

There are at least two formulations of this result: 
One is by dimension local semigroups, and the other is by $K_0$-groups. 
Although the former formulation appears in Elliott's original paper \cite{E}, it is seemingly less commonly used among specialists nowadays (indeed, it does not appear explicitly in standard textbooks like \cite{Da, BlK}). 
Let us briefly recall both of these two formulations because the former formulation will be essentially used in the proof of Theorem \ref{sumlattice}.

Let $A$ be an R$^*$-algebra. 
Let $\E(A)$ denote the collection of idempotents, i.e., $\E(A):=\{e\in A\mid e^2=e\}$. 
For $e, f\in \E(A)$, we write $e\sim_0 f$ if there exist $x, y\in A$ with $xy=e, yx=f$. 
Then $\sim_0$ is an equivalence relation. 
Set $d_0(A):= \E(A)/{\sim_0}$. 
For $e\in \E(A)$, let $[e]_0\in d_0(A)$ denote the equivalence class with $e\in [e]_0$. 
We define the partial binary operation $+$ in $d_0(A)$ in the following manner (\emph{partial} in the sense that the operation is defined only for a certain collection of pairs of elements){.} 
For $[e]_0, [f]_0\in d_0(A)$, $[e]_0+[f]_0$ is defined if and only if there are $e'\in [e]_0, f'\in [f]_0$ with $e'f'=f'e'=0$, and in that case, we define $[e]_0+[f]_0=[e']_0+[f']_0:=[e'+f']_0$. 
This is actually a well-defined partially defined binary operation. 

The pair $(d_0(A), +)$ can also be obtained by means of projections. 
Consider the set $d(A):=\P(A)/{\sim}$, where the equivalence relation $\sim$ means the Murray--von Neumann equivalence. 
For $p\in \P(A)$, let $[p]\in d(A)$ denote the equivalence class with $p\in [p]$. 
We define the partial operation $+$ in $d(A)$ in the following manner. 
For $[p], [q]\in d(A)$, $[p]+[q]$ is defined if and only if there are $p'\in [p], q'\in [q]$ with $p'q'=0$, and in that case, we define $[p]+[q]=[p']+[q']:=[p'+q']$. 
Then this partial operation is well-defined, and moreover, by a routine argument as in \cite{Kap} or \cite[Section 4]{BlK}, $(d(A), +)$ is isomorphic to $(d_0(A), +)$ as sets endowed with partially defined operations. 
The \emph{dimension local semigroup} of $A$ is the pair $(d(A), +)$. 
 
Below we give minimal information to obtain the $K_0$-group of an R$^*$-algebra. 
See for example \cite[Chapter III]{BlK} for more details about this. 
(Note that an R$^*$-algebra $A$ is a local C$^*$-algebra in the sense of \cite{BlK}, namely, $A$ is a pre-C$^*$-algebra that is closed under holomorphic functional calculus.)
For an R$^*$-algebra $A$, we consider the R$^*$-algebra $\M_{\infty}(A)$, which can be considered as the collection of $\N\times\N$ matrices with finitely many nonzero entries in $A$. 
For $n\geq 1$, we identify $\M_{n}(A)$ with the collection of all operators in $\M_{\infty}(A)$ whose nonzero entries are in the $n\times n$ upper-left corner.
In that way we identify $\M_{\infty}(A)$ with the directed union of $(\M_n(A))_{n\geq 1}$.

Observe that for every pair of projections $p, q\in \M_{\infty}(A)$ there exists a pair of mutually orthogonal projections $p', q'\in \M_{\infty}(A)$ with $p\sim p'$, $q\sim q'$. 
Thus the operation $+$ on the set $d(\M_{\infty}(A))$ is defined for every pair of projections. 
We define $V(A) := d(\M_{\infty}(A))$, which is an abelian semigroup endowed with operation $+$. 
The $K_0$-group $K_0(A)$ of $A$ is defined as the Grothendieck group of $(V(A), +)$.

In the setting of nonunital C$^*$-algebras, it is mathematically more natural to define the $K_0$-group using the $K_0$-group of the unitization (see \cite[Section 5]{BlK}). 
However, the above naive definition works in the setting of R$^*$-algebras: 
Recall that the projection lattice of an R$^*$-algebra (in particular $\M_{\infty}(A)$) forms a net, which is an approximate identity. 
Therefore, by \cite[Proposition 5.5.5]{BlK}, our definition of the $K_0$-group coincides with the usual definition.

Recall that a unital (C$^*$-, or more general normed) algebra is said to have \emph{stable rank one} if the collection of invertible elements is dense. 
Clearly, a unital R$^*$-algebra and its completion have stable rank one.
It follows by \cite[Proposition 6.5.1]{BlK} that $V(A)$ is a cancellation semigroup, which means that there is a canonical injection from $V(A)$ into $K_0(A)$. 
It is easy to see that the same is true for nonunital R$^*$-algebras.
The image of this injection is written as $K_0(A)_+$.
Then $K_0(A)_+$ forms a cone of $K_0(A)$ which makes $K_0(A)$ an ordered group.
The \emph{scale} of $K_0(A)$ is defined as the subset $\Sigma(A):=\{[p]\mid p\in \P(A) \,\,(= \P(\M_1(A))\subset\P(\M_{\infty}(A)))\}$ of $K_0(A)_+$.
We are now in the position of stating Elliott's theorem (see \cite{E}, \cite[Section 7]{BlK} and \cite[Chapter IV]{Da}). 

\begin{theorem}\label{elliott}
Let $A, B$ be two sequentially ultramatricial R$^*$-algebras. 
The following are equivalent. 
\begin{enumerate}
\item $A$ is $^*$-isomorphic to $B$. 
\item $d(A)$ is isomorphic to $d(B)$ as sets endowed with partially defined binary operations.
\item There exists an ordered group isomorphism from $K_0(A)$ onto $K_0(B)$ that sends $\Sigma(A)$ onto $\Sigma(B)$. 
\item $\widehat{A}$ is $^*$-isomorphic to $\widehat{B}$. 
\end{enumerate}
\end{theorem}

The equivalence $A\cong B \iff \widehat{A}\cong \widehat{B}$ means that distinct sequentially ultramatricial R$^*$-algebras have distinct completions (separable AF C$^*$-algebras), and moreover, for each separable AF C$^*$-algebra any pair of dense R$^*$-subalgebras with countable Hamel dimension is mutually $^*$-isomorphic. 
Recall that the same is true for the class of commutative R$^*$-algebras by Stone's theorem. 
Theorem \ref{elliott} is invalid for ultramatricial R$^*$-algebras with uncountable Hamel dimension. 
For example, it is easy to see that both
\[
(\widehat{F(\ell_2)}, d(F(\ell_2)), K_0(F(\ell_2)), K_0(F(\ell_2))_+, \Sigma(F(\ell_2)))
\]
and
\[
(\widehat{\M_{\infty}}, d(\M_{\infty}), K_0(\M_{\infty}), K_0(\M_{\infty})_+, \Sigma(\M_{\infty}))
\] 
are isomorphic to 
\[
(K(\ell_2), \Z_{\geq 0}, \Z, \Z_{\geq 0}, \Z_{\geq 0}).
\] 

In what follows we give a lattice theoretic condition that is equivalent to the conditions of Theorem \ref{elliott}. 
We say two projections $p, q$ of an R$^*$-algebra $A$ are \emph{perspective} if there exists a projection $r\in \P(A)$ such that $p\wedge r=0=q\wedge r$ and $p\vee r=q\vee r$. 
We write $p\approx q$ if there is a collection $p_0:=p, p_1, \ldots, p_n:=q$ of finitely many projections in $A$ such that $p_{j-1}$ and $p_{j}$ are perspective, $j=1, \ldots, n$.

\begin{lemma}\label{perpeq}
Let $A$ be an R$^*$-algebra. 
Let  $p, q\in \P(A)$. 
The following are equivalent. 
\begin{itemize}
\item $p\sim q$.
\item $p\vee q-p\sim p\vee q-q$.
\item $p\approx q$.
\end{itemize}
\end{lemma}
\begin{proof}
By restricting our attention to the unital R$^*$-subalgebra $(p\vee q)A(p\vee q)$, and applying \cite[Propositions 6.4.1 and 6.5.1]{BlK}, we see the equivalence of the first two items.
Suppose that the first two items hold. 
There is a partial isometry $v\in A$ such that $vv^*=p\vee q=v^*v$ and $vpv^* = q$. 
We may consider $v$ as a unitary operator in $(p\vee q)A(p\vee q)$. 
Since $v$ has finite spectrum, there is a (unique) positive operator $a\in (p\vee q)A(p\vee q)$ such that $\lVert a\rVert<2\pi$ and $v=e^{ia}$. 
Fix a sufficiently large $n$.
It is an elementary exercise to show that $p_{j-1}:= e^{ia(j-1)/n}pe^{-ia(j-1)/n}$ and $p_{j}:=e^{ia{j}/n}pe^{-iaj/n}$ are perspective (set e.g.\ $r:= p_{j-1}\vee p_{j}-p_{j-1}$), $j=1, \ldots, n$, showing that $p=p_0\approx p_n=e^{ia}pe^{-ia}=q$.

Let $r\in \P(A)$ be a projection such that $p\wedge r=0=q\wedge r$ and $p\vee r=q\vee r$. 
Then it is easy to see that the operator $x= q(p\vee r -r)p$ satisfies $r(x)=p$ and $l(x)=q$. 
By Lemma \ref{lr}, we have $p\sim q$. 
Since $\sim$ is transitive, we see that $p\approx q$ implies $p\sim q$.
\end{proof}

As a corollary{,} we see that a unital R$^*$-algebra has some special properties as a regular ring.
\begin{corollary}
A unital R$^*$-algebra $A$ is directly finite (i.e., if $x, y\in A$ and $xy=1$, then $yx=1$) and unit regular (i.e., for every $x\in A$ there exists an invertible $y\in A$ with $xyx=x$). 
\end{corollary}
\begin{proof}
Let $x\in A$.
By the preceding lemma combined with Lemma \ref{lr}, we see that $1-l(x)\sim 1-r(x)$. 
Take a partial isometry $v\in A$ with $vv^*=1-l(x)$ and $v^*v=1-r(x)$. 
A direct calculation shows that $x^{\dagger}+v^*$ has inverse $x+v$ and $x(x^{\dagger}+v^*)x=xx^{\dagger}x=x$. 
Thus $A$ is unit regular. 
Recall that unit regularity implies direct finiteness. 
Indeed, if $A$ is unit regular and $xy=1$, take invertible $a\in A$ with $xax=x$, then $xa=xaxy=xy=1$, thus $x=a^{-1}$ and $y=a$, which implies $yx=1$. 
\end{proof}

\begin{lemma}
Let $A$ be an R$^*$-algebra. 
Let  $p, q\in \P(A)$. 
Then $[p]+[q]$ is defined in $d(A)$ if and only if 
there exist projections $p', q'\in \P(A)$ with $p\sim p'$, $q\sim q'$ and $p'\wedge q'=0$, 
and moreover, if there exist such $p', q'$, then $[p]+[q]=[p'\vee q']$.
\end{lemma}
\begin{proof}
Assume that $p', q'\in \P(A)$ satisfy $p\sim p'$, $q\sim q'$ and $p'\wedge q'=0$. 
Then $x:= q' (p'\vee q' -p')$ satisfies $l(x)=q'$ and $r(x)=p'\vee q' -p'$, which implies that $q'\sim p'\vee q' - p'$. 
Thus we obtain $q\sim p'\vee q' - p'$, which together with $p\sim p'$ implies that $[p]+[q]$ is defined and equal to $[p' + (p'\vee q' - p')] = [p'\vee q']$.
The converse implication is trivial.
\end{proof}

\begin{theorem}\label{sumlattice}
Let $A, B$ be sequentially ultramatricial R$^*$-algebras.
If $\P(A)$ is lattice isomorphic to $\P(B)$, then $A$ is $^*$-isomorphic to $B$. 
\end{theorem}
\begin{proof}
By Theorem \ref{elliott} $(2)\Rightarrow(1)$ it suffices to show that we may construct $(d(A), +)$ only from the lattice structure of $\P(A)$. 
Observe that the relation $\approx$ is defined by means of lattice operations.  
It follows by Lemma \ref{perpeq} that the Murray--von Neumann equivalence relation is recovered from the lattice structure. 
Moreover, the preceding lemma implies that for a pair $p, q$ of projections, $[p]+[q]$ is defined if and only if there exist $p', q'\in \P(A)$ such that $p'\sim p$, $q'\sim q$, and $p'\wedge q'=0$, and in that case, we also have $[p]+[q]=[p'\vee q']$. 
This implies that the local semigroup structure of $d(A)$ can be fully recovered only from the lattice structure of $\P(A)$. 
Thus the proof is complete.
\end{proof}
In Section \ref{nd}, we will see that lattice isomorphisms of projection lattices of certain R$^*$-algebras can be described more explicitly by means of ring isomorphisms, as a consequence of von Neumann's theorem (Theorem \ref{neumann}). 

More results on separable AF C$^*$-algebras are summarized in \cite[Chapters III, IV]{Da}.
Mundici and Panti \cite{MP} studied separable AF C$^*$-algebras $\widehat{A}$ such that $d(A)$ forms a lattice.

\section{Von Neumann, Dye, Gleason}\label{nd}
In this section{,} we give several results concerning mappings on the lattice of projections of an R$^*$-algebra.

First{,} we give a consequence of von Neumann's result.
A \emph{ring isomorphism} is an additive and multiplicative bijection.
It is not necessarily linear. 
Indeed, there exist many nonlinear ring automorphisms of $\C$. 

Let $A, B$ be R$^*$-algebras. 
Recall that each principal right ideal of $A$ can be written as $pA$ for a unique projection $p\in \P(A)$. 
Therefore, if $\phi\colon A\to B$ is a ring isomorphism, then there exists a unique lattice isomorphism $\psi\colon \P(A)\to\P(B)$ such that 
\[
\phi(pA)=\psi(p)B
\]
for all $p\in \P(A)$. 
(Note that this equation implies $l(\phi(p))=\psi(p)$.)
Von Neumann's theorem states that the converse also holds under a certain condition.
\begin{theorem}[von Neumann's theorem for R$^*$-algebras]\label{neumann}
Let $A, B$ be unital R$^*$-algebras, and $n\geq 3$ an integer. 
If  $\psi\colon \P(\M_n(A))\to \P(B)$ is a lattice isomorphism, then there exists a ring isomorphism $\phi\colon \M_n(A)\to B$ such that 
\[
\phi(p\M_n(A))=\psi(p)B
\]
 for all $p\in \P(\M_n(A))$. 
\end{theorem}
\begin{proof}
This is a special case of \cite[Theorem II.4.2]{N}.
\end{proof}

The uniqueness of $\phi$ in the above theorem can be obtained from the proposition below. 
\begin{lemma}\label{rl}
Let $A\subset B(H)$ be an R$^*$-algebra and $e\in A$ an idempotent, i.e., $e^2=e$. 
If $r(e)\leq l(e)$, then $e$ is a projection.
\end{lemma}
\begin{proof}
By $e^2=e$, we have $l(e)=ee^{\dagger}=e^2e^{\dagger}=el(e)$. 
Since $r(e)\leq l(e)$, we have $el(e)=e$. 
Thus $e=l(e)$ is a projection.
\end{proof}

\begin{proposition}\label{uniqueri}
Let $A$ be an R$^*$-algebra, and let $n\geq 2$ be an integer. 
If $\phi\colon \M_n(A)\to \M_n(A)$ is a ring automorphism with
\[
\phi(p\M_n(A))=p\M_n(A)
\]
for all $p\in \P(\M_n(A))$, then $\phi$ is the identity mapping on $\M_n(A)$.
\end{proposition}
\begin{proof}
If $A$ is unital, then the desired conclusion is a special case of \cite[Theorem II.4.1]{N}. 
Let $p\in \P(\M_n(A))$. 
It is clear that $\phi(p)$ is an idempotent.
The equation $\phi(p\M_n(A))=p\M_n(A)$ implies $l(\phi(p))=p$. 
Moreover, for every $q\in \P(\M_n(A))$ with $pq=0$, we have $\phi(p)q\M_n(A) = \phi(p)\phi(q\M_n(A))=\phi(pq\M_n(A))=\{0\}$. 
Therefore, we have $\phi(p)q=0$ for every $q\in \P(\M_n(A))$ with $pq=0$. 
This means $r(\phi(p))\leq p=l(\phi(p))$.
The preceding lemma implies that $\phi(p)$ is a projection and hence $\phi(p)=l(\phi(p))=p$. 
It follows that $\phi$ restricts to a ring automorphism of $\M_n(qAq)$ for every $q\in \P(A)$.
Using the proposition for the unital case, we see that $\phi(x)=x$ for every $x\in \M_n(qAq)$. 
Since $\M_n(A)$ is the directed union of $(\M_n(qAq))_{q\in \P(A)}$, we have $\phi(x)=x$ for every $x\in \M_n(A)$.
\end{proof}

If $\phi$ and $\phi'$ are ring isomorphisms from $\M_n(A)$ onto $B$ that satisfy the condition of Theorem \ref{neumann}, then $\phi^{-1}\circ \phi'$ maps each principal right ideal of $\M_n(A)$ onto itself. 
Thus the preceding proposition implies $\phi^{-1}\circ \phi'=\operatorname{id}_{\M_n(A)}$, and $\phi'=\phi$.
A somewhat surprising fact is that we cannot drop the assumption of the existence of unit in Theorem \ref{neumann}.
\begin{example}\label{nonunital}
For an inner product space $V$, the lattice $\P(F(V))$ can be identified with the lattice of finite-dimensional subspaces of $V$ (ordered by inclusion). 
Therefore, if two inner product spaces $V, W$ have the same Hamel dimension, then $\P(F(V))$ is lattice isomorphic to $\P(F(W))$. 
Let $H$ be an infinite-dimensional Hilbert space, and let $f\colon H\to H$ be any linear bijection that is not bounded. 
Remark that $F(H)$ is $^*$-isomorphic to $\M_n(F(H))$ for every $n\geq 1$.
Since $f$ maps each finite-dimensional subspace of $H$ onto another finite-dimensional subspace of $H$, $f$ determines a lattice automorphism $\psi$ of the lattice of finite-dimensional subspaces of $H$, which is identified with $\P(F(H))$. 
Let us prove that there exists no ring automorphism $\phi\colon F(H)\to F(H)$ that satisfies $\phi(pF(H))=\psi(p)F(H)$ for all $p\in \P(F(H))$. 
Assume for a contradiction that there exists such a ring automorphism $\phi\colon F(H)\to F(H)$. 
By unboundedness of $f$, we may find a nonzero vector $h\in H$ such that $f(\{h\}^{\perp})$ is dense in $H$. 
Take the rank-one  projection $p\in \P(F(H))$ onto $\C h$. 
For each projection $q\in \P(F(H))$ of rank one with $pq=0$, we see that $\phi(q)\in F(H)$ is an idempotent with $\phi(p)\phi(q)=\phi(pq)=0$. 
However, the equation $\phi(qF(H))=\psi(q)F(H)$ implies $\phi(q) H = f(qH)$. 
It follows that  $\phi(p)|_{f(\{h\}^{\perp})}=0$. 
Since  $f(\{h\}^{\perp})$ is dense in $H$, the boundedness of $\phi(p)$ implies $\phi(p)=0$, a contradiction.
\end{example}

Recall that a mapping $\phi\colon A\to B$ between $^*$-algebras is called a \emph{Jordan $^*$-homomorphism} (resp.\ a \emph{Jordan $^*$-isomorphism}) if it is a linear mapping (resp. a linear bijection) such that $\phi(x)^*=\phi(x^*)$ and $\phi(x^2)=\phi(x)^2$ for every $x\in A$.
Clearly, if $A$ and $B$ are R$^*$-algebras, $\phi\colon A\to B$ is a Jordan $^*$-homomorphism and $x\in A$ is self-adjoint (resp.\ positive, a projection), then so is $\phi(x)\in B$.

Jordan $^*$-isomorphisms appear quite often in studying mappings between operator algebras that preserve certain structures. 
See Moln\'ar's recent article \cite{Mol} for more about Jordan $^*$-isomorphisms between C$^*$-algebras. 

\begin{example}
Let $A$ be the unitization of $\bigoplus_{n\in \Z}\M_2$, which is formed of all bi-infinite sequences $(x_n)_{n\in \Z}$ of $2\times 2$ matrices $x_n$ with the following property: There exists $N\geq 1$ such that $x_n=x_N\in \C \begin{pmatrix}1&0\\0&1\end{pmatrix}$ for all $n$ with $\lvert n\rvert\geq N$. 
Define $\phi\colon A\to A$ by $\phi((x_n)_{n\in \Z})= (y_n)_{n\in \Z}$, $y_n=x_n$ if $n\geq 0$, and $y_n$ is the transpose of $x_n$ otherwise.
Then $\phi$ is a Jordan $^*$-automorphism of $A$. 
\end{example}
\begin{proposition}\label{jordan}
Let $A, B$ be R$^*$-algebras. 
Let $\phi\colon A\to B$ be a linear map such that $\phi(\P(A))\subset\P(B)$.
Then $\phi$ is a Jordan $^*$-homomorphism. 
\end{proposition}
\begin{proof}
It is easy to check that $\phi$ maps mutually orthogonal projections to mutually orthogonal projections.
Since every self-adjoint operator $a$ in $A$ can be written as a real-linear combination of mutually orthogonal projections, the linearity of  $\phi$ implies that $\phi(a)^*=\phi(a)$ and $\phi(a^2)=\phi(a)^2$.
It follows that $\phi$ preserves involution. 
If $a, b$ are self-adjoint, then so is $a+b$, hence we have $\phi((a+b)^2)=\phi(a+b)^2$, which together with $\phi(a^2)=\phi(a)^2$ and $\phi(b^2)=\phi(b)^2$ implies that $\phi(ab+ba)=\phi(a)\phi(b)+\phi(b)\phi(a)$. 
Thus we obtain 
\[
\begin{split}
\phi((a+ib)^2) &= \phi(a^2)+i\phi(ab+ba)-\phi(b^2)\\
&=\phi(a)^2+i(\phi(a)\phi(b)+\phi(b)\phi(a))-\phi(b)^2\\
&=\phi(a+ib)^2.
\end{split}
\]
\end{proof}

By strengthening the assumption of Theorem \ref{neumann}, we obtain a Feldman--Dye type result. 
Let $A, B$ be R$^*$-algebras. 
A bijection $\psi\colon \P(A)\to \P(B)$ is called an orthoisomorphism if $pq=0\iff \psi(p)\psi(q)=0$ for every pair $p,q\in \P(A)$. 
It is easy to verify that an orthoisomorphism $\psi\colon \P(A)\to\P(B)$ is a lattice isomorphism and $\psi(p^{\perp})=\psi(p)^{\perp}$ if $A$ and $B$ are unital, so this definition is consistent with the definition of orthoisomorphism in Section \ref{com}.
Dye \cite{Dy} (see also Feldman's \cite{Fe}) proved that every orthoisomorphism between the projection lattices of two von Neumann algebras extends to a Jordan $^*$-isomorphism if one of the von Neumann algebras is without type I$_2$ direct summands. 
To obtain this theorem, Dye needed to appropriately modify the proof of von Neumann's \cite[Theorem II.4.2]{N}. 
In the setting of R$^*$-algebras, we may directly apply Theorem \ref{neumann} to consider orthoisomorphisms of the projection lattices.
The proof below is partly based on \cite[Proof of Theorem 3]{Fe}.
It should be worth mentioning that we do not need the assumption of the existence of unit in this theorem.

\begin{theorem}[Dye's theorem for R$^*$-algebras]\label{dye}
Let $A, B$ be  R$^*$-algebras. 
If $\phi\colon A\to B$ is a Jordan $^*$-isomorphism, then it restricts to an orthoisomorphism from $\P(A)$ onto $\P(B)$.
Let $n\geq 3$.
If $\psi\colon \P(\M_n(A))\to \P(B)$ is an orthoisomorphism, then there exists a unique Jordan $^*$-isomorphism $\phi\colon \M_n(A)\to B$ such that $\phi(p)=\psi(p)$ for all $p\in \P(\M_n(A))$. 
\end{theorem}
\begin{proof} 
It is clear that a Jordan $^*$-isomorphism restricts to an orthoisomorphism between the projection lattices.
In what follows, let $n\geq 3$ and suppose that $\psi\colon \P(\M_n(A))\to \P(B)$ is an orthoisomorphism.

Assume that $A$ is unital. 
Then $B$ is also unital because $\psi$ preserves the maximal projection.
Since $\psi$ is a lattice isomorphism, there exists a ring isomorphism $\phi\colon \M_n(A)\to B$ such that $\phi(p\M_n(A))=\psi(p)B$ for all $p\in \P(\M_n(A))$. 
Let $p\in \P(\M_n(A))$. 
It is clear that $\phi(p)$ is an idempotent. 
The equation $\phi(p^{\perp}\M_n(A))=\psi(p^{\perp})B = \psi(p)^{\perp}B$ implies that $l(\phi(p^{\perp}))=\psi(p)^{\perp}$. 
Since $pp^{\perp}=0$, we have  $\phi(p)\phi(p^{\perp}) = 0$. 
Thus we have $0=r(\phi(p))l(\phi(p^{\perp}))=r(\phi(p))\psi(p)^{\perp}$ and hence $r(\phi(p))\leq \psi(p)=l(\phi(p))$.
By Lemma \ref{rl}, we see that $\phi(p)$ is a projection, and we have $\phi(p)=l(\phi(p)) =\psi(p)\in \P(B)$.

Consider the ring automorphism $x\mapsto \phi^{-1}(\phi(x^*)^*)$ of $\M_n(A)$. 
This fixes every projection, hence Proposition \ref{uniqueri} implies that $x= \phi^{-1}(\phi(x^*)^*)$, or equivalently, $\phi(x)^*=\phi(x^*)$ for each $x\in \M_n(A)$.
It follows that $\phi$ maps the self-adjoint part of $\M_n(A)$ onto that of $B$. 
Since $\phi$ preserves squares, $\phi$ restricted to self-adjoint parts is an order isomorphism. 
It follows that $\phi(t)=t$ for all $t\in \R$. 
Define the complex-linear map $\Phi\colon A\to B$ by $\Phi(a+ib)=\phi(a)+i\phi(b)$ for self-adjoint $a, b\in A$. 
By Proposition \ref{jordan}, $\Phi$ is a Jordan $^*$-homomorphism that extends $\psi$.
Since projections of $B$ linearly span $B$, $\Phi$ is surjective.

If $A$ is not unital, then for each $p\in \P(A)$, $\psi$ restricts to an orthoisomorphism from $\P(\M_n(pAp))$ onto $\P(\psi(p\otimes 1)B\psi(p\otimes 1))$. 
Since $pAp$ is a unital R$^*$-algebra, it follows that $\psi$ restricted to $\P(\M_n(pAp))$ extends to a Jordan $^*$-isomorphism from $\M_n(pAp)$ onto $\psi(p\otimes 1)B\psi(p\otimes 1)$. 
Moreover, since $\P(\M_n(pAp))$ linearly spans $\M_n(pAp)$, this is the unique linear extension of $\psi|_{\P(\M_n(pAp))}$.
Since $\M_n(A)$ is the directed union of $(\M_n(pAp))_{p\in \P(A)}$, it follows that $\psi$ extends to a linear mapping, and this extension is a Jordan $^*$-isomorphism from $\M_n(A)$ onto $B$.
The uniqueness of the extension is again the consequence of the fact that an R$^*$-algebra is linearly spanned by projections.
\end{proof}

The assumption $n\geq 3$ is essential in  Theorem \ref{dye} (and Theorem \ref{neumann}).
Indeed, it is easy to find wild examples of orthoisomorphisms from $\P(\M_2)$ onto itself.

A \emph{$^*$-antiisomorphism} $\phi\colon A\to B$ between $^*$-algebras is a complex-linear bijection with $\phi(x^*)=\phi(x)^*$, $\phi(xy)=\phi(y)\phi(x)$ for every pair $x, y\in A$. 
Clearly, a $^*$-antiisomorphism is a Jordan $^*$-isomorphism.
Theorem \ref{dye} motivates the following question:
How far are Jordan $^*$-isomorphisms from $^*$-isomorphisms? 
In particular, how different are $^*$-antiisomorphisms from $^*$-isomorphisms?
{A m}ore specific question is:
Is there a pair of mutually $^*$-antiisomorphic R$^*$-algebras that are not $^*$-isomorphic, or equivalently, is there an R$^*$-algebra that is not $^*$-isomorphic to its opposite R$^*$-algebra?
Here, for an R$^*$-algebra $A$, its \emph{opposite R$^*$-algebra} $A^{\mathrm{op}}$ is defined in the following manner. 
As a complex normed space $A^{\mathrm{op}}$ is equal to $A$. 
For each $x\in A$ we write the corresponding element in $A^{\mathrm{op}}$ as $x^{\mathrm{op}}$ in order to distinguish $A$ from $A^{\mathrm{op}}$. 
We define the $^*$-algebra structure of $A^{\mathrm{op}}$ by $(x^{\mathrm{op}})^* := (x^*)^{\mathrm{op}}$ and $x^{\mathrm{op}}y^{\mathrm{op}} := (yx)^{\mathrm{op}}$ for $x, y\in A$. 
Then it is easy to show that $A^{\mathrm{op}}$ forms an R$^*$-algebra that is $^*$-antiisomorphic to $A$.

It is a trivial fact that every commutative R$^*$-algebra is $^*$-isomorphic to its opposite.
By the Elliott classification theorem{,} it is plain that every sequentially ultramatricial R$^*$-algebra is $^*$-isomorphic to its opposite. 
How about general R$^*$-algebras?
Banach space theory gives an answer. 
For a complex normed space $X$, its \emph{complex conjugate normed space} $\overline{X}$ is defined in the following manner. 
As a real normed space $\overline{X}$ is equal to $X$. 
For each $v\in X$ we write the corresponding element in $\overline{X}$ as $\overline{v}$ in order to distinguish $X$ from $\overline{X}$. 
We define the complex-linear structure of $\overline{X}$ by $\lambda \overline{v}:= \overline{\overline{\lambda}v}$ $v\in X$, $\lambda\in \C$. 

\begin{proposition}\label{opposite}
There is a purely atomic simple separable R$^*$-algebra $F(V)$ that is not $^*$-isomorphic to its opposite.
\end{proposition}
\begin{proof}
It is well-known (see for example \cite{Bo, K}) that there exists a separable Banach space $X$ that is not isomorphic (as complex Banach space) to its complex conjugate $\overline{X}$. 
For such an $X$, take an injective bounded linear operator $T\colon X\to \ell_2$ with dense range  (see Proposition \ref{range}), and set $V:= TX$. 
Consider the complex conjugate $\overline{T}\colon \overline{X} \to \overline{\ell_2}$ of $T$ (defined by $\overline{T}(\overline{v}):= \overline{T(v)},\,\, v\in X$), which is clearly a bounded linear operator with dense range, and observe that the range is canonically identified with $\overline{V}$. 
It is easy to see that the opposite R$^*$-algebra of $F(V)$ is $^*$-isomorphic to $F(\overline{V})$ by the mapping $x^{\mathrm{op}} \mapsto \overline{x}^*$, $x\in F(V)$. 
If $F(V)$ is $^*$-isomorphic to its opposite, then it follows by Proposition \ref{range} that $X$ and $\overline{X}$ are isomorphic as complex Banach spaces, which contradicts our choice of $X$. 
\end{proof}

The difference between an algebra and its opposite for operator algebras has been considered by many hands.
In the setting of von Neumann algebras, Connes \cite{C} discovered an example of a II$_1$ factor which is not $^*$-isomorphic to its opposite, giving a negative solution to Problem 4 of Kadison's list \cite{Ge}. 
By now many examples of C$^*$-algebras that are not $^*$-isomorphic to their opposites are known.
However, it remains as an open problem whether there exists a simple nuclear separable C$^*$-algebra that is not $^*$-isomorphic to its opposite. 
The case of  AF C$^*$-algebras is still under investigation in the literature. 
Indeed, in \cite{FH} Farah and Hirshberg proved the existence of a (nonseparable) AF C$^*$-algebra that is not $^*$-isomorphic to its opposite under some set{-}theoretic axiom that is consistent with ZFC. 
It is an open problem whether such an example can be constructed inside ZFC.
See \cite{FH} for more details about the state-of-the-art on this topic.

Let $A$ be an R$^*$-algebra or a von Neumann algebra. 
A function $\mu\colon \P(A)\to \R$ is called a \emph{finitely additive} (signed) \emph{measure} if it satisfies
\[
\mu(p+q) = \mu(p)+\mu(q)
\]
for every pair of mutually orthogonal projections $p, q\in\P(A)$. 

Mackey's problem asks whether a finitely additive measure extends to a linear mapping under a suitable assumption. 
This question is motivated by the theory of quantum logic. 
Gleason proved that every positive countably-additive measure on the projection lattice of $B(H)$, except for the case $H$ is $2$-dimensional,  extends to a bounded normal functional \cite{Gle}. 
More generally, it was proved by Bunce and Wright that every finitely additive measure with bounded range extends uniquely to a bounded linear functional in the setting of general von Neumann algebras without type I$_2$ direct summands \cite{BW1, BW2}.
We give a Gleason-type result for a certain class of R$^*$-algebras. 
Let $A$ be an ultramatricial R$^*$-algebra.
We say $A$ satisfies the property (N2) if $A$ can be written as the union of an increasing net $(A_i)_{i\in I}$ of finite-dimensional R$^*$-algebras such that no $A_i$ has $\M_2$ as its direct summand. 

Let $A$ be an R$^*$-algebra and $\mu$ a finitely additive measure on $\P(A)$.
Since every operator $x\in A$ has cartesian decomposition $x=a+bi$ with $a, b\in A$ self-adjoint, $x$ is written as a finite linear combination of projections. 
Thus the uniqueness of linear extension of a finitely additive measure is clear. 
If in addition $\lVert x\rVert\leq 1$, then  $0\leq a_{\pm}, b_{\pm}\leq 1$ (see Proposition \ref{operation} (2)) and they are in the convex hull of $\P(A)$. 
If $\mu$ extends to a linear functional $f$ on $A$, then we have $f(x) = f(a_+) -f(a_-) +if(b_+) -if(b_-)$, and $f(a_+), f(a_-), f(b_+), f(b_-)$ belong to the convex hull of $\mu(\P(A))$. 
Therefore, if $\mu(\P(A))$ is a bounded subset of $\R$ then the linear extension, if exists,  is bounded. 

\begin{theorem}[Gleason's theorem for ultramatricial R$^*$-algebras]\label{gleason}
Let $A$ be an ultramatricial R$^*$-algebra with (N2).
If $\mu\colon \P(A)\to \R$ is a finitely additive measure such that $\mu(\P(A))$ is a bounded subset of $\R$, then $\mu$ extends uniquely to a bounded linear functional on $A$. 
\end{theorem}
\begin{proof}
Assume that $A$ is the union of an increasing net $(A_i)_{i\in I}$ of finite-dimensional R$^*$-algebras such that no $A_i$ has $\M_2$ as its direct summand. 
Then Gleason's theorem for a finite-dimensional von Neumann algebra (see \cite[Section 2]{BW1}, which is based on Gleason's result \cite{Gle}) implies that the map $\mu$ restricted to each $\P(A_i)$ extends uniquely to a linear mapping on $A_i$. 
It follows that $\mu$ extends uniquely to a linear mapping on $A$, and the extension is bounded by the observation preceding this theorem.
\end{proof}

\begin{corollary}
Let $A$ and $B$ be R$^*$-algebras. 
Assume that $A$ is ultramatricial and satisfies (N2). 
If $\psi\colon \P(A)\to \P(B)$ is a map such that $\psi(p)+\psi(q)=\psi(p+q)$ for every pair of mutually orthogonal projections $p, q\in \P(A)$, then $\psi$ uniquely extends to a Jordan $^*$-homomorphism $\phi\colon A\to B$. 
\end{corollary}
\begin{proof}
For each bounded linear functional $f$ on $B$, the preceding theorem implies that the map $f\circ \psi\colon  \P(A)\to \C$ extends uniquely to a bounded linear functional of $A$. 
As a consequence, we may easily prove the following:
If complex numbers $z_1, z_2, \ldots, z_n\in \C$ and projections $p_1, \ldots, p_n\in \P(A)$ satisfy $\sum_{k=1}^n z_kp_k=0$, then $\sum_{k=1}^n z_k\psi(p_k)=0$. 
It follows that $\psi$ extends (uniquely) to a linear mapping $\phi\colon A\to B$. 
By Proposition \ref{jordan}, we conclude that $\phi$ is a Jordan $^*$-homomorphism.
\end{proof}

\section{Questions}\label{question}
In this section{,} we propose naive questions on R$^*$-algebras. 
These questions are not complete: some of them might be trivially easy to solve, and some might be totally meaningless to consider, but some should have {essential} importance.

First of all, the author suspects that the theory of R$^*$-algebras falls into one of the following two possibilities: 
\begin{itemize}
\item If every R$^*$-algebra is locally finite-dimensional or ultramatricial, then chances are that we can capture a lot about the structure of general R$^*$-algebras. Indeed, we already have a complete classification result of locally finite-dimensional R$^*$-algebras with countable Hamel dimension (Section \ref{countable}).
\item If not, then the theory of R$^*$-algebras should be so deep that the full understanding might be far out of the reach. 
That is because once we get one wild example of an R$^*$-algebra we may easily construct many other wild examples by Section \ref{example}.
In that case{,} R$^*$-algebras would be all the more interesting as a new area of ``noncommutative'' mathematics.
\end{itemize}

\begin{question}
Which of the above two is true?
If the first option is true, then is it possible to give a classification result that applies to a class broader than that in Section \ref{countable}?
(For example, to the class of separable R$^*$-algebras under some appropriate condition?)
If the second option is true, is it still possible to give a satisfactory classification theorem that is valid for all R$^*$-algebras with countable Hamel dimension?
\end{question}
Sz\H{u}cs and Tak\'acs proved that there exists a finitely generated infinite-dimensional R$^*$-algebra if and only if there exists a singly generated infinite-dimensional R$^*$-algebra $A$ (which means that there exists $x\in A$ with $\A^*(x)=A$) \cite[Proposition 4.7]{ST}. 
The latter remains to be a hard problem, it seems. 
The difficulty of considering singly generated $^*$-algebras can also be observed in the following relevant famous open problem, which is listed in Kadison's problems \cite{Ge} right before Problem \ref{kadison}. 
Is every von Neumann algebra acting on a separable Hilbert space $H$ singly generated (as a $\sigma$-weakly closed $^*$-subalgebra of $B(H)$)?
For more details on this famous problem, see \cite[Chapter 16]{SS}. 

We have seen the importance of the uniqueness of C$^*$-norm for R$^*$-algebras in Theorem \ref{c*norm}. 
Recall that uniqueness of C$^*$-norm is important in the theory of C$^*$-algebras, too. 
Indeed, nuclearity of a C$^*$-algebra, which lies in the core of the study of C$^*$-algebras, is defined in terms of the uniqueness of C$^*$-norm: 
A C$^*$-algebra $A$ is nuclear if for every C$^*$-algebra $B$ the algebraic tensor product $A\otimes B$ has a unique C$^*$-norm. 
It is natural to ask the following question. 
\begin{question}
Does the study of R$^*$-algebras give a better understanding of the uniqueness of C$^*$-norm? 
In particular, can we obtain a nice necessary/sufficient condition for a $^*$-algebra to have a unique C$^*$-norm? 
\end{question}

Property (LC$^*$) in Section \ref{defi} gives one sufficient condition.
In the case of the group algebra $\C[G]$ of a discrete group $G$, Grigorchuk, Musat and R{\o}rdam asked in \cite[Question 6.8]{GMR} whether $\C[G]$ has a unique C$^*$-norm if and only if $G$ is locally finite. 
It was pointed out by N. Ozawa that this question has a negative solution, see \cite{Al}. 
Let us also give a related question, which makes sense only in the case there exists a non-ultramatricial R$^*$-algebra. 
\begin{question}\label{tensor?}
Given a pair of non-ultramatricial R$^*$-algebras $A, B$, does $A\otimes B$ form an R$^*$-algebra?
\end{question}
Note that if the answer is affirmative then $A\otimes B$ has a unique C$^*$-norm, which further implies that $\widehat{A}\otimes \widehat{B}$ has a unique C$^*$-norm, or in other words, $(\widehat{A}, \widehat{B})$ is a nuclear pair \cite[Definition 9.1]{Pi}.
Another question {that} makes sense only in the case there exists a non-ultramatricial R$^*$-algebra is the following. 
\begin{question}
Does Theorem \ref{gleason} (Gleason-type theorem) generalize to a wider class of R$^*$-algebras? 
In particular, is the following statement true?
For every R$^*$-algebra $A$ and an integer $n\geq 3$, a finitely additive measure $\mu\colon \P(\M_n(A))\to \R$ with bounded range extends linearly to $\M_n(A)$.
\end{question}

Let us propose a question about the property of R$^*$-algebras that is shared with C$^*$-/von Neumann algebras. 
\begin{question}
Let $H$ be a Hilbert space. 
Study the class of $^*$-subalgebras $A\subset B(H)$ that satisfy the property of Proposition \ref{operation} (1). 
How about (2) and (3)? (In (2) we assume Borel measurability and boundedness of $f$.)
\end{question}
We know that every C$^*$-algebra and every R$^*$-algebra satisfy (1).
So does the union of an increasing net of C$^*$-algebras.  
In Blackadar's book \cite{BlK} a $^*$-subalgebra $A\subset B(H)$ that satisfies (1) is called a local C$^*$-algebra.
Note that in the article \cite{BH} the same term local C$^*$-algebra is used in another sense: 
It means a $^*$-subalgebra $A\subset B(H)$ with the property (LC$^*$) in Section \ref{defi}, which is a condition far weaker than (2).
Every von Neumann algebra and every R$^*$-algebra satisfy (2) and (3). 
If $A\subset B(H)$ is a $^*$-subalgebra and the $\sigma$-weak closure of $\A^*(x)$ is contained in $A$ for every normal element $x\in A$, then $A$ satisfies (2).   
We do not know whether (2) implies (1).

Recall that a directed union of R$^*$-algebras is an R$^*$-algebra. 
This together with Zorn's lemma enables us to consider maximal R$^*$-algebras: 
For every C$^*$-algebra $A$ there is a maximal R$^*$-subalgebra $B\subset A$ (in the sense that if $B_0\subset A$ is an R$^*$-subalgebra of $A$ with $B\subset B_0$ then $B=B_0$).  
\begin{question}
Is it possible to find a concrete example of a maximal R$^*$-subalgebra of $B(\ell_2)$ (without resorting to Zorn's lemma)? 
\end{question}
Remark that if $A\subset B(H)$ is an R$^*$-subalgebra then $A+F(H)+\C \operatorname{id}_H$ is also an R$^*$-algebra. 
Thus a maximal R$^*$-subalgebra $A\subset B(H)$ needs to satisfy $F(H)+\C \operatorname{id}_H\subset A$.

By the type decomposition theorem of von Neumann algebras and Dye's theorem, it is easy to see that a pair of von Neumann algebras have mutually orthoisomorphic projection lattices if and only if the von Neumann algebras are Jordan $^*$-isomorphic. 
We do not have a succinct decomposition result for R$^*$-algebras. 
That is why the following question can possibly have a negative answer. 
\begin{question}
Let $A, B$ be two R$^*$-algebras. 
Assume that $\P(A)$ and $\P(B)$ are orthoisomorphic. 
Does it follow that $A$ is Jordan $^*$-isomorphic to $B$?
\end{question}
It is also of interest whether the ``orthoisomorphic'' in this question can be replaced with ``lattice isomorphic''. 
Example \ref{nonunital} implies that this is not true in the non-unital case.
On the other hand, the author's result \cite[Corollary 4.3]{M} (see also \cite[Corollary 1.5]{AK} by Ayupov and Kudaybergenov) implies that a pair of certain von Neumann algebras 
have mutually isomorphic projection lattices (as lattices) if and only if the von Neumann algebras are Jordan $^*$-isomorphic.
If we restrict ourselves to the case $A$ and $B$ are sequentially ultramatricial R$^*$-algebras, Theorem \ref{sumlattice} states that $A$ and $B$ are $^*$-isomorphic if $\P(A)$ is lattice isomorphic to $\P(B)$. 
This motivates a further question concerning the non-ultramatricial case: 
\begin{question}
Let $A, B$ be R$^*$-algebras with countable Hamel dimension such that $\P(A)$ is lattice isomorphic to $\P(B)$.
Does it follow that $A$ is $^*$-isomorphic to $B$? 
\end{question}

Every statement about (generalized) Boolean algebras gives rise to a question about R$^*$-algebras. 
Below we propose several questions in this direction. 

Before that, we remark that there are several other attempts to give a framework for the theory of noncommutative Boolean algebras or noncommutative Boolean spaces. 
One natural way is to detect the class of commutative C$^*$-algebras that arise as $C_0(K)$ for zero-dimensional (locally) compact Hausdorff spaces $K$. 
This class coincides with that of commutative C$^*$-algebras with real rank zero  (i.e., every self-adjoint operator is in the norm closure of the collection of self-adjoint operators with finite spectra). 
Thus one may consider C$^*$-algebras with real rank zero as noncommutative analogs of zero-dimensional locally compact Hausdorff spaces.
See \cite[Section V.3]{Bl} for more information about real and stable rank.

Another way of formulating noncommutative Boolean algebras is to consider a certain class of inverse semigroups, which in connection with groupoids can be related to the theory of C$^*$-algebras. 
See \cite{L} for the details. 
Yet another direction of research is the theory of quantum logic initiated by Birkhoff and von Neumann \cite{BN}, 
see \cite{Dv}.
The main target of this theory is an orthomodular lattice (usually assumed to be $\sigma$-complete), which is an orthocomplemented lattice $\P$ such that $p\vee (q\wedge p^{\perp})=q$ for any $p, q\in \P$ with $p\leq q$.
This class of lattices include{s} the projection lattices of von Neumann algebras.

By Stone's theorem{,} zero-dimensional locally compact Hausdorff spaces are in one-to-one correspondence with generalized Boolean algebras. 
Recall that our main discovery in this paper is that R$^*$-algebras form noncommutative analogs of generalized Boolean algebras. 
For this reason{,} the author guesses that those C$^*$-algebras obtained by completing R$^*$-algebras can be considered as noncommutative analogs of zero-dimensional locally compact Hausdorff spaces.
(However, we need to be aware of the fact that the completion, e.g.\ $K(H)$, can have many different dense R$^*$-subalgebras.)
Now it is natural to ask:
\begin{question}
Which C$^*$-algebra arises as the completion of an R$^*$-algebra?
\end{question}
By definition (see Example \ref{af}) this class contains all AF C$^*$-algebras.
Recall that the completion of an R$^*$-algebra has real rank zero, and stable rank one if unital.
Is it true that the completion of every R$^*$-algebra is nuclear?
If every R$^*$-algebra is ultramatricial, then a C$^*$-algebra is the completion of some R$^*$-algebra if and only if it is AF. 
Blackadar once wrote in the first edition \cite[7.1]{BlK1} of his book on K-theory of C$^*$-algebras the sentence ``AF algebras are in some sense the ``zero-dimensional" C$^*$-algebras'', but he added to this sentence in the second edition \cite[7.1]{BlK} the phrase ``although real rank zero is now regarded as the most natural noncommutative analog of zero-dimensionality''.

Let us summarize several facts on Boolean algebras.
Although Theorem \ref{elliott} is very important in the theory of AF C$^*$-algebras, it gives no better understanding of commutative AF C$^*$-algebras/generalized Boolean algebras. 
This fact can be observed in the following proposition.
\begin{proposition}
Let $\S\subset 2^X$ be a generalized field of sets. 
As in Section \ref{com} we identify $R(\S)$ with the space of functions $\{f\colon X\to \C\mid \#f(X)<\infty,\,\,f^{-1}(z)\in \S\text{ for all }z\in \C\setminus\{0\}\}$.
\begin{itemize}
\item $R(\S)$ has countable Hamel dimension  if and only if $R(\S)$ is separable if and only if $\S$ has countable cardinality.
\item $K_0(R(\S))$ is isomorphic to $\{f\in R(\S)\mid f(X)\subset\Z\}$, and $K_0(R(\S))_+$ to $\{f\in K_0(R(\S))\mid f\geq 0\}$.
\item $d(R(\S))$ is the same as $\S$ as a set, and for $A, B\in d(R(\S))=\S$, the sum $A+B$ is defined if and only if $A\wedge B=0$, in which case we have $A+B=A\vee B$.
\end{itemize}
\end{proposition}
We omit the proof because it is easy and these facts should be well-known. 
There exist many wild examples of Boolean algebras and Boolean spaces, which implies the existence of wild commutative R$^*$-algebras.

\begin{example}\label{ketonen}
In \cite{Ke}, as a byproduct of his structural theorem for countable Boolean algebras, Ketonen proved that there exists a countable Boolean algebra $\S$ such that $\S$ is lattice isomorphic to $\S\times \S\times \S$ but not to $\S\times \S$. 
This is known as a counterexample to {Tarski's} cube problem.
It follows that the commutative unital R$^*$-algebra $R(\S)$ has countable Hamel dimension, and $R(\S)$ is $^*$-isomorphic to $R(\S)\oplus R(\S)\oplus R(\S)$ but not to $R(\S)\oplus R(\S)$. 
\end{example}

\begin{example}\label{trmkova}
It is known that every metrizable Boolean space $X$ that is homeomorphic to $X\times X\times X$ is homeomorphic to $X\times X$. 
It is also known that there exists a separable Boolean space $X$ that is homeomorphic to $X\times X\times X$ but not to $X\times X$. 
See Trnkov\'a's \cite[2.5]{Tr}. 
Recall that for a locally compact Hausdorff space $K$, $K$ is metrizable if and only if $C_0(K)$ is separable.
It follows that there exists no countable Boolean algebra $\S$ such that $R(\S)$ is $^*$-isomorphic to $R(\S)\otimes R(\S)\otimes R(\S)$ but not to $R(\S)\otimes R(\S)$, and there does exist such an example if we drop the assumption of countability.
\end{example}

Are there similar examples in the noncommutative case? 

\begin{question}\label{rcube}
Is there a simple or countable-dimensional R$^*$-algebra $A$ that is $^*$-isomorphic to $A\otimes A\otimes A$, but not to $A\otimes A$? 
\end{question}
Question \ref{tensor?} might force us to consider only the case of ultramatricial R$^*$-algebras.
See also the Scottish Book \cite{SB} Problem 77. 

Let us also consider the following similar question: 
Is there an R$^*$-algebra $A$ that is $^*$-isomorphic to $\M_3(A)$, but not to $\M_2(A)$? 
If we restrict ourselves to the simple purely atomic case, this reduces to the following one. 
Is there an inner product space $V$ such that $V$ is isometrically isomorphic to the $\ell_2$-direct sum $V\oplus V\oplus V$ but not to $V\oplus V$?
This question is reminiscent of an example by Gowers and Maurey in \cite{GM} of a Banach space $X$ that is isomorphic to $X\oplus X\oplus X$ but not to $X\oplus X$ (as Banach spaces). 
In the general setting of R$^*$-algebras, the infinite tensor product $\bigotimes_{n\geq 1} \M_3$ is an example that satisfies the condition of the question.
(The author thanks Masaki Izumi for pointing out this fact.)

In \cite[Theorem 4.1]{KR} Kalton and Roberts proved that there exists a constant $K>0$ such that for every Boolean algebra $\S$ and every map $f\colon \S \to \R$ with $f(0)=0$ and $\lvert f(p\vee q)-f(p)-f(q)\rvert \leq 1$ for any mutually orthogonal pair $p, q\in \S$, there exists a finitely additive (signed) measure $\mu\colon \S\to \R$ such that $\lvert f(p)-\mu(p)\rvert\leq K$ for every $p\in \S$. 
This theorem was applied in \cite{KR} to study a certain property of  (quasi-)Banach spaces. 
By comparing this with Theorem \ref{gleason}, it is fairly natural to formulate its noncommutative version. 
\begin{question}[Noncommutative analog of the Kalton--Roberts theorem]
Is there a constant $K>0$ with the following property?
For every unital R$^*$-algebra $A$ and every map $f\colon \P(A) \to \R$ with $f(0)=0$ and $\lvert f(p\vee q)-f(p)-f(q)\rvert \leq 1$ for any mutually orthogonal pair $p, q\in \P(A)$, there exists a finitely additive (signed) measure $\mu\colon \P(A)\to \R$ such that $\lvert f(p)-\mu(p)\rvert\leq K$ for every $p\in \P(A)$. 
\end{question}
It might be interesting to consider the same question for the projection lattice of a von Neumann algebra.

\end{document}